\documentclass[reqno,final]{amsart}
\usepackage{natbib}  
\usepackage{fancyhdr} 
\usepackage{color} 

\usepackage{hyperref}
\hypersetup{
colorlinks=false,       
linkcolor=red,          
citecolor=green,       
filecolor=magenta,   
urlcolor=cyan     }

\usepackage{graphicx} 
\usepackage[font=small]{caption}

\usepackage{pstricks}
\usepackage{amssymb}

\usepackage[utf8]{inputenc}
\usepackage[T1]{fontenc}

\usepackage{amssymb,amsmath,amscd}
\usepackage{mathrsfs}
\usepackage{mathtools} 
\usepackage{dsfont} 
\usepackage{bbm}

\usepackage[mathscr]{euscript}

\usepackage{subfig}

\usepackage{tabularx} 
\usepackage{calc} 

\usepackage{appendix}

\usepackage{enumerate}
\usepackage[shortlabels]{enumitem} 
\setitemize{label=$\bullet$}

\definecolor{aleacolor}{rgb}{0.16,0.59,0.78}

\hypersetup{
breaklinks,
colorlinks=true,
linkcolor=aleacolor,
urlcolor=aleacolor,
citecolor=aleacolor}

\renewcommand{\headheight}{11pt}
\renewcommand{\cite}{\citet}

\theoremstyle{plain}
\newtheorem{theorem}{Theorem}[section]                                          
\newtheorem{proposition}[theorem]{Proposition}                          
\newtheorem{lemma}[theorem]{Lemma}

\theoremstyle{definition}

\theoremstyle{remark}
\newtheorem{remark}[theorem]{Remark}

\newcounter{claimcounter}
\newenvironment{claim}{\refstepcounter{claimcounter}{\textbf{Claim \theclaimcounter:}}}{}

\makeatletter \@addtoreset{equation}{section} \makeatother
\renewcommand\theequation{\thesection.\arabic{equation}}
\renewcommand\thefigure{\thesection.\arabic{figure}}
\renewcommand\thetable{\thesection.\arabic{table}}

\newcommand{\N}{\mathbb{N}}

\newcommand{\R}{\mathbb{R}}

\newcommand{\D}{\mathbb{D}}

\newcommand{\Pro}{\mathbb{P}}

\newcommand{\F}{\mathbb{F}} 
\newcommand{\f}{\mathcal{F}} 

\newcommand{\U}{\mathbb{U}} 
\newcommand{\tree}{\mathcal{U}} 

\newcommand{\Ctree} {\mathbb{T}} 
\newcommand{\ctree}{\mathcal{T}} 

\newcommand{\Y}{Y}

\newcommand{\Co}{\mathcal{C}} 
\newcommand{\K}{\mathcal{K}}

\newcommand{\e}{\mathrm{e}}
\newcommand{\der}{\mathrm{d}}
\newcommand{\1}{\mathbbm{1}}

\begin{document}

\title{Time reversal dualities for some random forests}




\author{Miraine Dávila Felipe}
\author{Amaury Lambert}

\address{UMR 7599, Laboratoire de Probabilit\'es et Mod\`eles Al\'eatoires, UPMC and  CNRS,\newline Case courrier 188, 4, Place Jussieu, 75252 PARIS Cedex 05}

\address{UMR 7241, Centre Interdisciplinaire de Recherche en Biologie, Collège de France,\newline 11, place Marcelin Berthelot, 75231 PARIS Cedex 05.}

\email{miraine.davila\_felipe@upmc.fr , amaury.lambert@upmc.fr}
\urladdr{\url{http://www.proba.jussieu.fr/pageperso/amaury/index.htm}}



\begin{abstract}
We consider a random forest $\mathcal F^*$, defined as a sequence of i.i.d. birth-death (BD) trees, each started at time 0 from a single ancestor, stopped at the first tree having survived up to a fixed time $T$.  We denote by $\left(\xi^*_t,\ 0\leq t\leq T\right)$ the population size process associated to this forest, and we prove that if the BD trees are supercritical, then the time-reversed process $\left(\xi^*_{T-t},\ 0\leq t\leq T\right)$, has the same distribution as $\left(\widetilde\xi^*_t,\ 0\leq t\leq T\right)$, the corresponding population size process of an equally defined forest $\widetilde{\mathcal F}^*$, but where the underlying BD trees are subcritical, obtained by swapping birth and death rates. 
We generalize this result to splitting trees (i.e. life durations of individuals are not necessarily exponential), provided that the i.i.d. lifetimes of the ancestors have a specific explicit distribution, different from that of their descendants.
The results are based on an identity between the contour of these random forests truncated up to $T$  and the duality property of Lévy processes, which allow to derive other useful properties of the forests with potential applications in epidemiology.
\end{abstract}
\maketitle

\section{Introduction}

We consider a model of branching population in continuous time, where individuals behave independently from one another. They give birth to identically distributed copies of themselves at some positive rate throughout their lives, and have generally distributed lifetime durations. 
A splitting tree \cite{GeKe97,Lam10} describes the genealogical structure under this model and the associated population size process  is a so-called (binary, homogeneous) Crump-Mode-Jagers (CMJ) process.
When the lifetime durations are exponential or infinite (and only in this case) this is a Markov process, more precisely, a linear birth-death (BD) process.

Here, trees are assumed to originate at time 0 from one single ancestor. For a fixed time $T>0$, we define a forest $\f^*$ as a sequence of i.i.d. splitting trees, stopped at the first one having survived up to $T$, and we consider the associated population size or width process $\left(\xi^*_t,\ 0\leq t\leq T\right)$. 
In the case of birth-death processes we have the following identity in distribution .

\begin{theorem}\label{theo:intro}
Let $\f^*$ be a forest as defined previously, of supercritical  birth-death trees with parameters $b>d>0$, 
then the time-reversed process satisfies
\[
\left(\xi^*_{T-t},\ 0\leq t\leq T\right)\,{\buildrel d \over =}\,\left(\widetilde\xi^*_t,\ 0\leq t\leq T\right)
\]
where the right-hand side is the width process of an equally defined forest, but where the underlying trees are subcritical, obtained by swapping birth and death rates (or equivalently, by conditioning on ultimate extinction \citealp{AthNe72}).
\end{theorem}

We further generalize this result to splitting trees, provided that the (i.i.d.) lifetimes of the ancestors have a specific distribution, explicitly known and different from that of their descendants. This additional condition comes from the \textit{memory} in the distribution of the lifespans when they are not exponential, that imposes a distinction between ancestors and their descendants as we will see in Section~\ref{sec:results}.
To our knowledge, this is the first time a result is established that reveals this kind of duality in branching processes: provided the initial population is structured as described before, the width process seen backward in time is still the population size process of a similarly defined forest.

Furthermore, these dualities through an identity in distribution are established not only for the population size processes, but for the forests themselves. In other words, we give here the construction of the dual forest $\widetilde\f^*$, from the forest $\f^*$, by setting up different filiations between them, but where the edges of the initial trees remain unchanged. This new genealogy has no interpretation, so far, in terms of the original family, and can be seen as the tool to reveal the intrinsic branching structure of the backward-in-time process. 
\newline

The results are obtained via tree contour techniques and some properties of Lévy processes.
The idea of coding the genealogical structure generated by the branching mechanism through a continuous or jumping stochastic process has been widely exploited with diverse purposes by several authors, see for instance \cite{Po04,GeKe97,GaJa98,DuGa02,Lam10,BaPaBa13}. 

Here we make use of a particular way of exploring a splitting tree, called \textit{jumping chronological contour process} (JCCP).
We know from \cite{Lam10} that this process has the law of a  spectrally positive Lévy process properly reflected and killed. 
The notion of JCCP can be naturally extended to a forest by concatenation. Then our results are proved via a pathwise decomposition of the contour process of a forest and space-time-reversal dualities for Lévy processes \cite{Ber92}. 
We define first some path transformations of the contour of a forest $\f^*$, after which, the reversed process will have the law of the contour of a forest $\widetilde\f^*$. The invariance of the \textit{local time} (defined here as the number of times the process hits a fixed value of its state-space $\R_+$) of the contour after these transformations, allows us to deduce the aforementioned identity in distribution between the population size processes.
\newline

Branching processes are commonly used in biology to represent, for instance, the evolution of individuals with asexual reproduction \cite{Jag91,KiAx02}, or a group of species \cite{NMH94,Sta09}, as well as the spread of an epidemic outbreak in a sufficiently large susceptible population \cite{Bec74,Bec76,Tan06}. We are particularly interested in the last application, which was the primary motivation for this work, and where our duality result has some interesting consequences.

When modeling epidemics, we specify that what was called so far a \textit{birth event} should be thought of as a transmission event of the disease from one (infectious) individual to another (susceptible, assumed to be in excess). In the same way a \textit{death event} will correspond to an infectious individual becoming non-infectious (will no longer transmit the pathogen, e.g. recovery, death, emigration, etc.). Then the branching process describes the dynamics of the size of the infected population, and the splitting tree encodes the history of the epidemic.

In the last few decades, branching processes have found many applications as stochastic individual-based models for the transmission of diseases, especially the Markovian case (notice however, that the assumption of exponentially distributed periods of infectiousness is mainly motivated by mathematical tractability, rather than biological realism).
In recent years, the possibility of sequencing pathogens from patients has been constantly increasing, and with it, the interest in using the phylogenetic trees of pathogen strains to infer the parameters controlling the epidemiological mechanisms, leading to a new approach in the field, the so-called phylodynamic methods \cite{Gren04}. A very short review on the subject is given later in Section~\ref{sec:epi}.

Here we consider the situation where the data consists in incidence time series (number of new cases registered through time) and the \textit{reconstructed transmission tree} (i.e. the information about non sampled hosts is erased from the original process). These trees are considered to be estimated from pathogen sequences from present-time hosts. These observed statistics are assumed to be generated from a unique forward in time process and since no further hypothesis is made, they are not independent in general.
Hence, the computation of the likelihood as their joint distribution quickly becomes a delicate and complex issue, even in the linear birth-death model, since it requires to integrate over all the possible extinct (unobserved) subtrees between 0 and $T$.

Therefore, to solve this problem, we propose a description of the population size process $I\coloneqq (I_t,0\leq t\leq T)$, conditional on the reconstructed genealogy of individuals that survive up until time $T$ (i.e. the reconstructed phylogeny).
This result is a consequence of the aforementioned duality between random forests $\f^*$ and $\widetilde\f^*$.
We state that, under these conditions, the process $I$, backward in time, is the sum of the width processes of independent birth-death trees, each conditioned on its height to be the corresponding time of coalescence, plus an additional tree conditioned on surviving up until time $T$.

The structure in the population, that is the definition of our forests, and the fact that in our model the sampled epidemic comes all from one single ancestor at time 0, can be thought of as a group of strains of a pathogen in their attempts to invade the population, but where only one succeeds (at time $T$). 
However, if various invading strains succeed, analogous results can be deduced by concatenating (summing) an equal number of forests. The general assumption will be then, that for each \textit{successful} strain, there is a geometric (random) number of other strains of the pathogen that become extinct before time $T$. 
The probability of success of these geometric r.v. depends on the recovery and transmission parameters. 
Finally, estimating these parameters from molecular and epidemiological data using this branching processes model, can be addressed through MLE or Bayesian inference. However, these statistical questions are not directly treated here. 
\medskip

For the sake of clarity, and in order to be consistent with most of the literature on branching processes, we will prefer to use the terms of birth and death events throughout the document, except when we present the outcomes of our results in the specific context of epidemiology in Section~\ref{sec:epi}.
The rest of the paper is organized as follows: Section~\ref{sec:preliminaries} is dedicated to some preliminaries on Lévy processes, trees, forests and their respective contours. Finally, in Section~\ref{sec:results}, we state our main results and give most of their proofs, although some of them are left to the Appendix.

\section{Preliminaries}\label{sec:preliminaries}

\subsubsection*{Basic notation}
Let $E=\R\cup\{\partial\}$ where $\partial$ is a topologically isolated point, so-called \textit{cemetery point}. Let $\mathcal{B}(E)$ denote the Borel $\sigma$-field on $E$. 
Consider the space $\D(\R_+,E)$ (or simply $\D$) of càdlàg functions $\omega$ from $\R_+$ into the measurable space $(E,\mathcal{B}(E))$ endowed with Skorokhod topology \cite{JaShi03}, stopped upon hitting $\partial$ and denote the corresponding Borel $\sigma$-field by $\mathcal{B}(\D)$. 
Define the \textit{lifetime} of a path $\omega \in \D$ as $\zeta=\zeta(\omega)$, the unique value in $\R_+\cup\{+\infty\}$ such that $\omega(t)\in \R$ for $0\leq t< \zeta$, and $\omega(t)=\partial$ for every $t\geq\zeta$.
Here $\omega(t-)$ stands for the left limit of $\omega$ at $t\in\R_+$, $\Delta \omega(t)= \omega(t)-\omega(t-)$ for the size of the (possible) jump at $t\neq zeta$ and we make the usual convention $\omega(0-)=\omega(0)$. 

We consider stochastic processes, on the probability space $(\D,\mathcal{B}(\D),\mathbf{P})$, say $X=\left(X_t,t\geq 0\right)$, also called the coordinate process, having $X_t=X_t(\omega)=\omega(t)$. In particular, we consider only processes with no-negative jumps, that is such that $\Delta X_t \in\R_+$ for every $t\geq 0$.
The canonical filtration is denoted by $(\mathscr{F}_t)_{t\geq 0}$.
 
Let $\mathcal{P}(E)$ be the collection of all probability measures on $E$. We use the notation  $\mathbf{P}_x(X\in\cdot ) = \mathbf{P} \left( X\in\cdot \middle\vert X_0 = x\right)$ and for $\mu\in\mathcal{P}(E)$, 
\[
\mathbf{P}_\mu(X\in\cdot )\coloneqq \int\limits_E \mathbf{P}_x(X\in\cdot ) \mu(\der x).
\]
For any measure $\mu$ on $[0,\infty]$, we denote by $\overline\mu$ its tail, that is
\[
\overline{\mu}(x) \coloneqq \mu([x,+\infty])
\]

Define by $\tau_A \coloneqq \inf\{t> 0: X_t\in A\}$, the first hitting time of the set $A\in\mathcal{B}(E)$, with the conventions $\tau_x=\tau_{\{x\}}$, and $\tau_x^-=\tau_{(-\infty,x)}$, $\tau_x^+=\tau_{(x,+\infty)}$ for any $x\in\R$.

\subsubsection*{Some path transformations of càdlàg functions}
In this subsection we will define some deterministic path transformations and functionals of càdlàg stochastic processes.
Define first the following classical families of operators acting on the paths of $X$:
\begin{itemize}
\item the \textit{shift operators}, $\theta_s, s\in \R_+$, defined by
\[
\theta_s(X_t)\coloneqq X_{s+t}, \qquad \forall t\in \R_+
\]
\item the \textit{killing operators}, $k_s$, $s\in\R_+$ , defined by
\[
k_s(X) \coloneqq \left\{
	\begin{array}{ll}
		X_t & \mbox{if } s< t\ \\
		\partial & \mbox{otherwise}
	\end{array}
\right.
\]
the killing operator can be generalized to killing at random times, for instance $X\circ k_{\tau_A}=k_{\tau_A}(X) = k_{\tau_A(X)}(X)$, denotes the process \textit{killed at the first passage into $A$}. It is easy to see that if $X$ is a Markov process, so is $X\circ k_{\tau_A}$.

\item the \textit{space-time-reversal operator} $\rho_s$, $s\in\R_+^*$, as
\[
\rho_s(X)_t \coloneqq X_0 - X_{(s-t)-} \qquad \forall t\in[0,s)
\]
and we denote simply by $\rho$ the space-time-reversal operator at the lifetime  of the process, when $\zeta<+\infty$, that is
\[
\rho(X)_t  \coloneqq X_0 - X_{(\zeta-t)-} \qquad \forall t\in[0,\zeta)
\]
\end{itemize}
The notations $\mathbf{P}\circ \theta_s^{-1}$, $\mathbf{P}\circ k_s^{-1}$ and $\mathbf{P}\circ \rho^{-1}$ stand for the law of the shifted, killed and space-time-reversed processes when $\mathbf{P}$ is the law of $X$.

When $X$ has finite variation we can define its local time process, taking values in $\N\cup\{+\infty\}$ as follows,
\[
\Gamma_r(X) = \textrm{Card}\{t\geq 0: X_t=r\}, \quad r\in \R,
\]
that is the number of times the process hits $r$ (before being sent to $\partial$).

For a sequence of processes in the same state space, say $(X_i)_{i\geq 1}$ with lifetimes $(\zeta_i, i\geq 1)$, we define a new process by the concatenation of the terms of the sequence, denoted by
\[
 [X_1,X_2,\dots]
\]
where the juxtaposition of terms is considered to stop at the first element with infinite lifetime. For instance, if $\zeta_1<+\infty$ and $\zeta_2=+\infty$, then for every $n\geq 2$
\[
 [X_1,X_2, \ldots, X_n]_t = 
\left\{
	\begin{array}{ll}
		X_{1,t}& \mbox{if } 0\leq t<\zeta_1 \\
		X_{2,t-\zeta_1} & t \geq \zeta_1
	\end{array}
\right.
\]
\medskip

We consider now a \textit{clock} or time change that was introduced in \cite{Ber96,Don07} in order to construct the probability measure of a Lévy process conditioned to stay positive, that will have the effect of erasing all the subpaths of $X$ taking non-positive values and closing up the gaps. 
We define it here for a càdlàg function $X$, for which we introduce the time it spends in $(0,+\infty)$, during a fixed interval $[0,t]$, for every $t\leq \zeta(X)$,
\[A_t\coloneqq\int\limits_0^t \1_{\{X_u> 0\}} \der u\]
and  its right-continuous inverse, $\alpha (t)\coloneqq \inf\left\{u\geq 0: A_u>t\right\}$, such that a time substitution by $\alpha$, gives a function with values in $[0,+\infty)\cup\{\partial\}$, in the following sense,
\[
(X\circ \alpha)_t = 
\left\{
	\begin{array}{ll}
		X_{\alpha(t)} & \mbox{if } \alpha(t)<+\infty \\
		\partial & \mbox{otherwise}
	\end{array}
\right.
\]

\begin{remark}\label{rem:reflect} 
Analogous time changes $\alpha^s$ (or $\alpha_s$) can be defined for any $s\in\R$, removing the excursions of the function above (or below) the level $s$, that is by time-changing $X$ via the right-continuous inverse of $t\longmapsto\int_0^t \1_{\{X_u < s \}} \der u$ (or $t\longmapsto\int_0^t \1_{\{X_u > s \}} \der u$).
\end{remark}

\medskip

\paragraph{Last passage from $T$ to 0:}
Fix $T>0$. 
Define the following special points for $X$
\begin{eqnarray}
g_T &\coloneqq& \inf\{t>0:X_t=T \}\nonumber \\
g_0 &\coloneqq& \inf\{t>g_T: X_{t-}=0\} \nonumber \\ 
\overline{g}_T &\coloneqq& \sup\{t<g_0: X_t=T\} \nonumber
\end{eqnarray}
and suppose $g_0<+\infty$.
We want to define a transformation of $X$ so that the subpath in the interval $[\overline{g}_T,g_0)$ will be placed at the beginning, shifting to the right the rest of the path.
Therefore, we define the function $\vartheta(X)=\vartheta:[0,g_0]\rightarrow[0,g_0]$ as,
\begin{eqnarray}
\vartheta(t)\coloneqq \left\{
	\begin{array}{ll}
		\overline{g}_T + t  & \mbox{if } \ 0 \leq t < g_0 - \overline{g}_T\\
		t - (g_0 - \overline{g}_T) & \mbox{if } \ g_0- \overline{g}_T \leq t \leq g_0
	\end{array}
\right. \nonumber
\end{eqnarray}
and then we consider the transformed path $\chi (X)$, defined for each $s\geq 0$ as,
\begin{eqnarray}
\chi(X)_s \coloneqq \left\{
	\begin{array}{ll}
		X_{\vartheta(s)}  & \textrm{if } s\in [0,g_0) \\
		\partial & \textrm{if } s\geq g_0
	\end{array}
\right. \nonumber
\end{eqnarray}

\subsection{Spectrally positive Lévy processes}\label{subsec:spect-pos-Levy}

\textit{Lévy processes} are those stochastic processes with stationary and independent increments, and almost sure right continuous with left limits paths.
During this subsection we recall some classic results from this theory and establish some others that will be used later. We refer to \cite{Ber96} and \cite{Ky06} for a detailed review on the subject.
\newline

We consider a real-valued Lévy process $\Y=\left(\Y_t,t\geq 0\right)$
and we denote by $P_x$ its law conditional on $\Y_0=x$. We assume $\Y$ is \textit{spectrally positive}, meaning it has no negative jumps. This process is characterized by its Laplace exponent $\psi$, defined for any $\lambda\geq 0$ by
\[
E_0\left[\e^{-\lambda \Y_t}\right] = \e^{t\psi(\lambda)}.
\]
We assume, furthermore, that $\Y$ has finite variation. Then $\psi$ can be expressed, thanks to the Lévy-Khintchine formula as
\begin{eqnarray}
\psi(\lambda) = -d\lambda - \int\limits_0^\infty \left(1-\e^{-\lambda r}\right)\Pi(\der r), 
\end{eqnarray}
where $d\in\R$ is called \textit{drift coefficient} and $\Pi$ is a $\sigma$- finite measure on $(0,\infty]$, called the \textit{Lévy measure}, satisfying $\int_0^{\infty}(r\wedge 1)\Pi(\der r)<\infty$. Notice we allow $\Pi$ to charge $+\infty$, which amounts to killing the process at rate $ \Pi(\{+\infty\})$.
Some might prefer to say, in other words, that $\Y$ is a subordinator (increasing paths) with possibly negative drift and possibly killed at a constant rate.

The Laplace exponent is infinitely differentiable, strictly convex (when $\Pi \not \equiv 0$), $\psi(0)=0$ (except when $\Pi$ charges $+\infty$, in which case $\psi(0)=\psi(0+)=-\Pi(\{+\infty\})$) and $\psi(+\infty)=+\infty$.
Let $\eta\coloneqq \sup\{\lambda\geq 0: \psi(\lambda)=0\}$. Then we have that $\eta=0$ is the unique root of $\psi$, when  $\psi^\prime(0+)=1-m\geq 0$. Otherwise the Laplace exponent has two roots, 0 and $\eta>0$. 
It is known that for any $x>0$,
\[P_x\left(\tau_0<+\infty\right)=\e^{-\eta x}.
\]
More generally, there exists a unique continuous increasing function $W:[0,+\infty)\rightarrow[0,+\infty)$, called the \textit{scale function}, characterized by its Laplace transform,
\[
\int\limits_0^{+\infty} \e^{-\lambda x}W(x) \der x = \dfrac{1}{\psi(\lambda)}, \qquad \lambda>\eta,
\]
such that for any $0<x<a$,
\begin{equation} \label{eq:exit-two-side}
P_x\left(\tau_0<\tau_a^+\right) = \dfrac{W(a-x)}{W(x)} 
\end{equation}

\subsubsection*{Pathwise decomposition}

For a Markov process, a point $x$ of its state space is said to be \textit{regular} or \textit{irregular} for itself, if $\mathbf{P}_x(\tau_x=0)$ is 1 or 0.
In a similar way, when the process is real-valued, we can say it is \textit{regular downwards} or \textit{upwards} if we replace $\tau_x$ by $\tau_x^-$ or $\tau_x^+$ respectively.
For a spectrally positive Lévy processes we know from \cite{Ber96} that $\Y$ has bounded variation if and only if 0 is irregular (and irregular upwards).
In this case there is a natural way of decomposing the process into excursions from any point $x$ on its state space.
For simplicity, here we depict the situation for $x=0$, the generalization being straightforward from the Markov property.

The process $\Y$ under $P_0$, can be described as a sequence of independent and identically distributed excursions from $0$, stopped at the first one with infinite lifetime.
Define the sequence of the successive hitting times of $0$, say $(\tau_0^{(i)}, i\geq 0)$ with $\tau_0^{(0)}=0$. For $i\geq 0$, on $\{\tau_0^{(i)}<+\infty\}$, define the shifted process,
\[
\epsilon_i \coloneqq 	\left( \Y_{\tau_0^{(i)}+t}, 0\leq t< \tau_0^{(i+1)} - \tau_0^{(i)} \right).
\]
The strong Markov property and the stationarity of the increments imply that $(\epsilon_i,i\geq 0)$ is a sequence of i.i.d. excursions, all distributed as $(\Y_t, 0\leq t\leq \tau_0)$ under $P_0$, with a possibly finite number of elements, say $N+1$, which is geometric with parameter $P_0(\tau_0=+\infty)$, corresponding to the time until the occurrence of an infinite excursion.
Hence, the paths of $\Y$ are structured as the juxtaposition of these i.i.d. excursions, $N$ with finite lifetime, followed by a final infinite excursion.

We now introduce, for any $a>0$, the process $\Y$ reflected below $a$, that is the process being immediately restarted at $a$ when it enters $(a,+\infty)$. This process is also naturally decomposed in its excursions below $a$, in the same way described before.
More precisely, let $(\epsilon_i,i\geq 1)$ be a sequence of i.i.d. excursions distributed as $\Y$ under $P_a$, but killed when they hit $(a,+\infty)$, that is, with common law $P_a\circ k^{-1}_{\tau_a^+}$. Notice that since the process is irregular upwards, then necessarily these excursions have a strictly positive lifetime $P_a$-a.s..
Define the reflected process $Y^{(a)}$ as their concatenation, that is
\begin{eqnarray}\label{eq:reflect}
Y^{(a)} \coloneqq [\epsilon_1,\epsilon_2,\ldots]
\end{eqnarray}

\subsubsection*{Overshoot and undershoot}\label{subsec:over-undershoot}
Formula (8.29) from \cite{Ky06} adapted to spectrally positive Lévy processes gives the joint distribution of the overshoot and undershoot of $\Y$ when it first enters the interval $[a,+\infty]$, without hitting 0. Let $x\in(0,a)$,  then for $u\in(0,a]$ and $v\in(0,+\infty)$, 
\begin{eqnarray}
&&P_x\left((a-\Y_{\tau_a^+-})\in \der u,(\Y_{\tau_a^+}-a)\in \der v, \tau_a^+<\tau_0\right)  \nonumber \\
&&\qquad \qquad = \left(\dfrac{W(a-x)W(a-u)}{W(a)}-W(a-x-u)\right)\der u\Pi(\der v+u) \label{eq:under-overshoot-0}
\end{eqnarray}

In the same way, if the restriction on the minimum of the process before hitting $[a,+\infty)$ is removed, choosing $x=a=0$ for simplicity, we have for $u,v\in(0,+\infty)$
\begin{eqnarray}\label{eq:under-overshoot}
P_0\left((\Y_{\tau_0^+-})\in \der u,\Y_{\tau_0^+}\in \der v,\tau_0^+<+\infty\right)  = \e^{-\eta u}\der u\Pi(\der v+u)
\end{eqnarray}

Then we can compute the distribution of the undershoot and overshoot  of an excursion away from 0, of the process $Y$ starting at $0$ on the event $\tau_0^+<+\infty$. By integrating (\ref{eq:under-overshoot}) we get,
\begin{eqnarray}
P_0\left(-\Y_{\tau_0^+-}\in \der u,\tau_0^+<+\infty\right) &=& \der u \int\limits_{[0,+\infty]} \e^{-\eta u}\Pi(\der v+u)  = \e^{-\eta u} \overline{\Pi}(u) \der u \nonumber \\
P_0\left(\Y_{\tau_0^+}\in \der v,\tau_0^+<+\infty\right) &=& \e^{\eta v} \der v \int\limits_{[v,+\infty]}  \e^{-\eta y} \Pi(\der y)  = \e^{\eta v} \overline{\widetilde\Pi}(v) \der v \nonumber 
\end{eqnarray}
where we define the measure 
\[
\widetilde\Pi(\der y)\coloneqq \e^{-\eta y} \Pi(\der y)
\]
on $(0,+\infty)$ of mass $\widetilde{b}:=b-\eta$.

Further, we can deduce from Theorem VII.8 in \cite{Ber96} that
\begin{eqnarray}
P_0\left(\tau_0^+<+\infty \right) = 1-\dfrac{W(0)}{W(\infty)} = 1\wedge m \label{eq:return-to-zero-finite}
\end{eqnarray}
as a consequence of (\ref{eq:exit-two-side}) and
\begin{eqnarray}
\lim_{t\rightarrow+\infty}W(t)
= \left\lbrace \begin{array}{cc}
+\infty & \textrm{if} \ \  m \geq 1 \\
\frac{1}{1-m} &  \textrm{if} \ \ m < 1
\end{array} \right., \qquad W(0)= 1. \nonumber
\end{eqnarray}
These two formulas come from the analysis of the behavior at $0$ and $+\infty$ of $W$'s Laplace transform, $1/\psi$, followed by the application of a Tauberian theorem. We refer to Propositions 5.4 and 5.8 from \cite{Lam10} for the details. 
\newline

We denote by $\mu_\top$ the probability measure on $[0,\infty)$ of the undershoot away from 0 under $P_0\left(\cdot\middle\vert \tau_0^+<+\infty\right)$, defined as follows,
\begin{eqnarray}
\mu_\top(\der u)\coloneqq  P_0\left(-\Y_{\tau_0^+-}\in \der u\middle\vert \tau_0^+<+\infty\right) = \dfrac{\e^{-\eta u} \overline{\Pi}(u)}{m\wedge 1} \der u\label{eq:undershoot}
\end{eqnarray}

Analogously the probability distribution of the corresponding overshoot is,
\begin{eqnarray}
\mu_\bot(\der v) \coloneqq P_0\left(\Y_{\tau_0^+}\in \der v\middle\vert \tau_0^+<+\infty\right) = \dfrac{\e^{\eta v} \overline{\widetilde\Pi}(v)}{m\wedge 1} \der v  \label{eq:overshoot}
\end{eqnarray}

\begin{remark}
In the exponential case with rates $b$ and $d$ it is not hard to see that the overshoot is exponentially distributed with parameter $d$, and the undershoot with parameter $b\vee d$.
\end{remark}

\subsubsection*{Process conditioned on not drifting to $+\infty$}

The following statements are direct consequences of Corollary VII.2  and Lemma VII.7 from \cite{Ber96}. 
The process $\Y$ drifts to $-\infty$, oscillates or drifts to $+\infty$, if $\psi^\prime(0+)$ is respectively positive, zero, or negative. 
As we mentioned before, only in the first case the Laplace exponent $\psi$ has a strictly positive root $\eta$, which leads to considering a new family of probability measure $\widetilde P_x$ via the exponential martingale $(\e^{-\eta (\Y_t-x)},t\geq 0)$, that is with Radon-Nikodym density
\[
\left.\frac{\der \widetilde P_x}{\der P_x} \right\vert_{\mathscr{F}_t} =\e^{-\eta (\Y_t-x)}.
\]
As we will show later, this can be thought of as the law of the initial Lévy process conditioned to not drift to $+\infty$, and in fact, $\Y$ under $\widetilde P$ is still a spectrally positive Lévy process with Laplace exponent 
\[
\widetilde\psi(\lambda) = \psi(\eta + \lambda) = -d\lambda - \int\limits_0^\infty \left(1-\e^{-\lambda r}\right)\widetilde\Pi(\der r),
\]
which is the Laplace exponent of a finite variation, spectrally positive Lévy process with drift $d$ and Lévy measure $\widetilde\Pi$. 
Furthermore, it is established that for every $x>0$, the law of the process until the first hitting time of $0$, that is $\left(\Y_t, 0\leq t< \tau_0\right)$, is the same under $P_x(\cdot\vert\tau_0<+\infty)$ as under $\widetilde P_x(\cdot)$, that is
\begin{eqnarray}\label{eq:prob-cond-hit-0}
P_x(\cdot\vert\tau_0<+\infty)\circ k^{-1}_{\tau_0} = \widetilde P_x\circ k^{-1}_{\tau_0} 
\end{eqnarray}

\medskip

Finally, we compute the following convolution products, which will be used hereafter, obtained from a direct inversion of the Laplace transform for $W$ and $\widetilde W$ (see p. 204-205 in \cite{Ber96}), where $\widetilde{W}$ is the scale function defined with respect to the Laplace exponent $\widetilde\psi$,
\begin{eqnarray}
\int\limits_0^T  W(T-v) \overline{\Pi}(v) \der v  = W(T) - 1 \label{eq:convol-W} \\
\int\limits_0^T  \widetilde{W}(T-v) \overline{\widetilde \Pi}(v) \der v  = \widetilde W(T) -1 \label{eq:convol-W-tilde}
\end{eqnarray}

\subsubsection*{Time-reversal for Lévy processes}

Another classical property of Lévy processes is their duality under time-reversal in the following sense: if a path is space-time-reversed at a finite time horizon, the new path has the same distribution as the original process.
More precisely, in the case of spectrally positive Lévy processes we have the following results from \cite{Ber92}:

\begin{proposition}[Duality]\label{prop:duality}
The process $\Y$ has the following properties: 
\begin{enumerate}[label={\upshape(\roman*)}]
\item 
under $P_0\left(\cdot\middle\vert \tau_0<+\infty\right)$, $\left(\Y_{t}, 0\leq t< \tau_0\right)$  and $\left(-\Y_{(\tau_0-t)-}, 0\leq t< \tau_0\right)$ have the same law
\item 
under $P_0\left(\cdot\middle\vert -\Y_{\tau_0^+-}= u\right)$ the reversed excursion,
$\left(-\Y_{(\tau_0^+-t)-}, 0\leq t< \tau_0^+\right)$ has the same distribution as $\left(\Y_{t}, 0\leq t< \tau_0\right)$ under $P_u\left(\cdot\middle\vert\tau_0<+\infty\right)$
\item 
under $P_0\left(\cdot\middle\vert -\Y_{\tau_0^+-}= y, \Delta Y (\tau_0^+) = z\right)$, the processes $\left(-\Y_{(\tau_0^+-t)-}, 0\leq t< \tau_0^+\right)$ and $\left(\Y_{\tau_0^++t}, 0\leq t< \tau_0-\tau_0^+\right)$ are independent. The first one has law $P_y (\cdot\vert \tau_0<+\infty) \circ k^{-1}_{\tau_0}$ and the second one has law $P_{z-y}\circ k^{-1}_{\tau_0}$
\end{enumerate}
\end{proposition}

\subsection{Trees and forests}

We refer to \cite{Lam10} for the rigorous definition and properties of discrete and chronological trees. The notation here may differ from that used by this author, so it will be specified in the following, as well as the main features that will be subsequently required.

A \textit{discrete tree}, denoted by $\tree$, is a subset of $\U\coloneqq\bigcup_{n\geq 0}\N^n$, satisfying some specific well known properties.
From a discrete tree $\tree$ we can obtain an $\R$-tree by adding birth levels to the vertices and lengths (\textit{lifespans}) to the edges, getting what is called a \textit{chronological tree} $\ctree$. 
For each individual $u$ of a discrete tree $\tree$, the associated \textit{birth level} is denoted by $\alpha(u)$ and the \textit{death level} by $\omega(u)$ ($\alpha(u)\in\R_+$, $\omega(u)\in\R_+\cup\{+\infty\}$ and such that $\alpha(u)<\omega(u)$).

Then $\ctree$ can be seen as the subset of $\U\times [0,+\infty)$ containing all the \textit{existence points} of individuals (vertices of the discrete tree): for every $u\in\tree$, $s\in [0,+\infty)$, then $(u,s)\in\ctree$ if and only if $\alpha(u)< s \leq \omega(u)$. The root will be denoted by $\rho\coloneqq(\varnothing,0)$ and $\pi_1$ and $\pi_2$ stand respectively for the canonical projections on $\U$ and $ [0,+\infty)$.
$\Ctree$ denotes the set of all chronological trees.

For any individual $u$ in the discrete tree $\tree=\pi_1(\ctree)$, we denote by $\zeta(u)$ its lifespan, i.e. $\zeta(u) \coloneqq \omega(u)-\alpha(u)$. 
Then the \textit{total length} of the chronological tree is the sum of the lifespans of all the individuals, that is,  \[\ell(\ctree)\coloneqq\sum_{v\in\pi_1(\ctree)}\zeta(v)\leq +\infty .\]

We will also refer to the \textit{truncated tree} up to level $s$, denoted by $\ctree^{(s)}$ for the chronological tree formed by the existence points $(u,t)$ such that $t\leq s$.
A chronological tree is said to be \textit{locally finite} if for every level $s\in[0,+\infty)$ the total length of the truncated tree is finite, $\ell(\ctree^{(s)})<+\infty$.

We can define the \textit{width} or \textit{population size process} of locally finite chronological trees as a mapping $\Xi$ that maps a chronological tree $\ctree$ to the function $\xi:\R_+\rightarrow \N$ counting the number of extant individuals at time $t\geq 0$
\begin{eqnarray}
\Xi(\ctree) \coloneqq  \left(\xi_t(\ctree), t \geq 0\right) 	
\end{eqnarray}
where for every $t\geq 0$,
\begin{eqnarray}
\xi_t(\ctree) = \textrm{Card}\{v\in\pi_1(\ctree):\alpha(v)<t\leq \omega(v)\} \nonumber
\end{eqnarray}
These functions are càdlàg, piecewise constant, from $\R_+$ into $\N$, and are absorbed at 0. 
Then we can define the time of extinction of the population in a tree as $T_{\textrm{Ext}}\coloneqq \inf\{t\geq 0: \xi_t(\ctree)=0\} $. 
Notice there might be an infinite number of jumps in $\R_+$, but only a finite number in every compact subset in the case of locally finite trees.

Chronological trees are assumed to be embedded in the plane, as on Fig.~\ref{fig:contour} (right), with time running from bottom to top, dotted lines representing filiations between individuals: the one on the left is the parent, and that on the right its descendant.

We will call a \textit{forest} every finite sequence of chronological trees, and we will denote the set of all forests by $\F$. More specifically for every positive integer $m$, let us define the set of \textit{$m$-forests} as follows,
\[\F_m\coloneqq \{(\ctree_1,\ctree_2,\ldots,\ctree_m): \ctree_1,\ctree_2,\ldots,\ctree_m\in\Ctree\}\]
then,
\[\F = \bigcup_{m\in\N} \F_m\]
It is straightforward to extend the notion of width process to a forest, say $\f=(\ctree_1,\ctree_2,\ldots,\ctree_m)\in\F$, as the sum of the widths of every tree of the sequence, i.e.,
\[\Xi(\f) : = \sum\limits_{i=1}^m \Xi(\ctree_i)\]

\subsection{The contour process}\label{subsec:contour}

As mentioned before, the genealogical structure of a chronological tree can be coded via continuous or càdlàg functions. They are usually called contour or exploration processes, since they refer mostly to deterministic functions of a (randomly generated) tree. In this case, the contour is a $\R$-valued stochastic process containing all the information about the tree, so that the latter can always be recovered from its contour. Among the different ways of exploring a tree we will exploit the jumping chronological contour process (JCCP) from \cite{Lam10}.

The JCCP of a chronological tree $\ctree$ with finite length $\ell=\ell(\ctree)$, denoted by $\Co(\ctree)$, is a function from $[0,\ell]$ into $\R_+$, that starts at the lifespan of the ancestor  and then walks backward along the right-hand side of this first branch at speed $-1$ until it encounters a birth event, when it jumps up of a height of the lifespan of this new individual, getting to the next tip, and then repeating this procedure until it eventually hits $0$, as we can see in Fig.~\hyperref[contour]{\ref*{fig:contour}} (see \cite{Lam10} for a formal definition).

Then it visits all the existence times of each individual exactly once and satisfies that the number of times it hits a time level, say $s\geq 0$, is equal to the number of individuals in the population at time $s$. 
More precisely, for any finite tree $\ctree$
\[
\Gamma\circ \Co(\ctree) = \Xi(\ctree),
\]
and more generally, if $\ctree$ is locally finite, this is also satisfied for the truncated tree at any level $s>0$, that is
\[
\Gamma\circ \Co(\ctree^{(s)}) = \Xi(\ctree^{(s)}).
\]
We can extend the notion of contour process to a forest $\f=(\ctree_1,\ctree_2,\ldots,\ctree_m)$ of finite total length $\overline{\ell} \coloneqq\sum_{i=1}^{m}\ell(\ctree_i)$, similarly to the way it is done in \cite{DuGa02}, by concatenating the contour functions,
\[
\left[\Co_t(\ctree_1), t\in[0,\ell(\ctree_1))\right), \ \left(\Co_t(\ctree_2), t\in[0,\ell(\ctree_2))\right),\ldots, \left(\Co_t(\ctree_m), t\in[0,\ell(\ctree_m))\right].
\] 
It will be denoted as well by $\Co(\f)$ or simply $\Co$ when there is no risk of confusion. We notice that the function thus obtained determines a unique sequence of chronological trees since they all start with one single ancestor.

\begin{center}
\begin{figure}
\includegraphics[scale=.7]{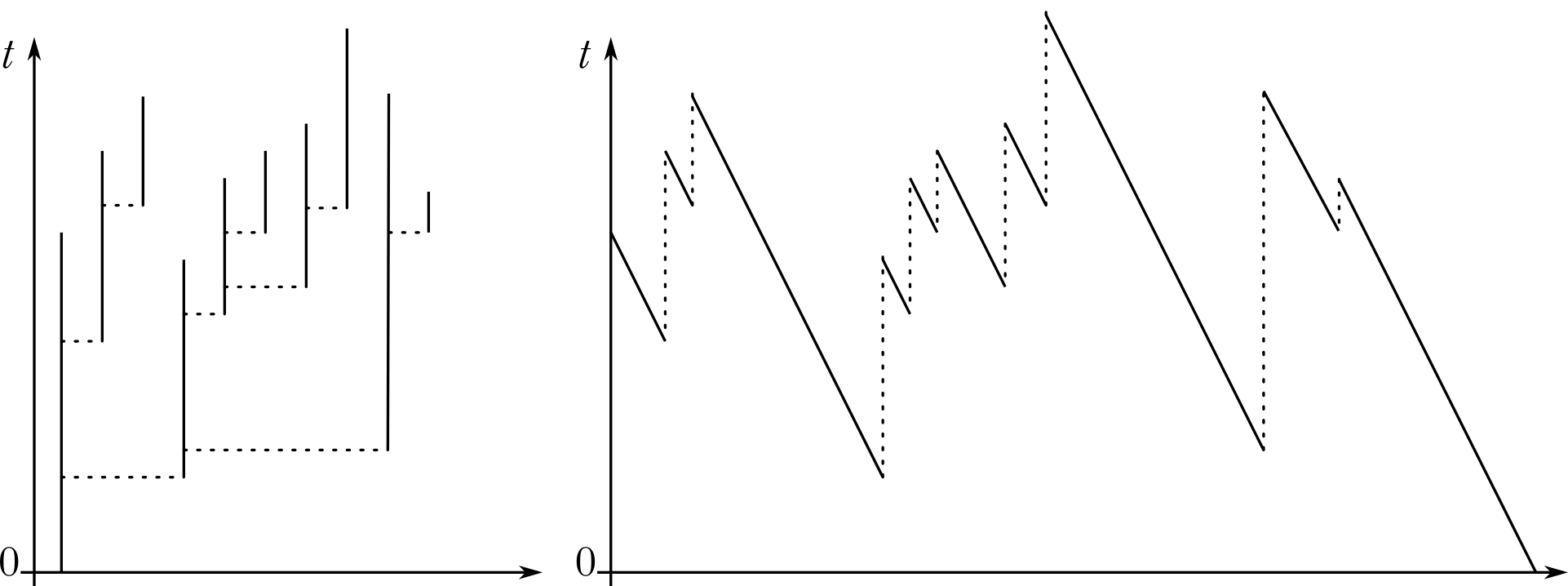}
\caption{An example of chronological tree with finite length (left) and its contour process (right).}
\label{fig:contour}
\end{figure}
\end{center}

\subsection{Stochastic model}
We consider a population (or particle system) that originates at time $0$ with one single progenitor. Then individuals (particles) evolve independently of each other, giving birth to i.i.d. copies of themselves at constant rate, while alive, and having a life duration with general distribution. 
The family tree under this stochastic model will be represented by a \textit{splitting tree}, that can be formally defined as an element $\ctree$ randomly chosen from the set of chronological trees, characterized by a $\sigma$-finite measure $\Pi$ on $(0,\infty]$ called the \textit{lifespan measure}, satisfying $\int_{(0,\infty]}\left(r\wedge 1\right)\Pi(\der r)<\infty$.

In the general definition individuals may have infinitely many offspring. However, for simplicity we will assume that $\Pi$ has mass $b$, corresponding to a population where individuals have i.i.d. lifetimes distributed as $\Pi(\cdot)\slash b$ and give birth to single descendants throughout their lives at constant rate $b$, all having the same independent behavior. In most of the following results this hypothesis is not necessary, and they remain valid if $\Pi$ is infinite.

Under this model the width process $\Xi(\ctree) =\left(\xi_t(\ctree),t\geq 0\right)$ counting the population size through time is a \textit{binary homogeneous Crump-Mode-Jagers process} (CMJ). This process is not Markovian, unless $\Pi$ is exponential (\textit{birth-death process}) or a Dirac mass at $\{+\infty\}$ (\textit{Yule process}). 

A tree, or its width process $\Xi$, is said to be subcritical, critical or supercritical if
\[m\coloneqq\int\limits_0^{+\infty}r\Pi(\der r)\]
is respectively less than, equal to or greater than 1, and we define the extinction event $\textrm{Ext} \coloneqq \{\lim_{t\rightarrow\infty} \xi_t\left(\ctree\right)=0\}$.

For a splitting tree we can define, as well as for its deterministic analogue, the JCCP.
Actually, the starting point of the present work, is one of the key results in \cite{Lam10}, where the law of the JCCP of a splitting tree truncated up to $T$ or conditional on having finite length is characterized by a Lévy process.

\begin{theorem}[\cite{Lam10}]\label{theo:lambert-2010}
If $X^{(T)}$ is the JCCP of a splitting tree with lifespan measure $\Pi$ truncated up to $T\in (0,+\infty)$ and $\Y$ is a spectrally positive Lévy process with finite variation and Laplace exponent $\psi(\lambda)=\lambda - \int_0^\infty (1-\exp(-\lambda r))\Pi(\der r), \ \lambda\geq 0$, then conditional on the lifespan of the ancestor to be $x$, $X^{(T)}$ has the law of $\Y$, started at $x\wedge T$, reflected below $T$ and killed upon hitting $0$. 
Furthermore, conditional on extinction, 
$X$ has the law of $\Y$ started at $x$, conditioned on, and killed upon hitting $0$.
\end{theorem}

We state here without proof the following elementary lemma that is repeatedly used in the proofs.
\begin{lemma}\label{lemma:geom}
Let $Z$ be a r.v. in a probability space $(\Omega,\mathscr{F},\mathbf{P})$, taking values in a measurable space $(E,\mathcal{A})$.
Let $A\in\mathcal{A}$ be such that $p\coloneqq \mathbf{P}(Z\in A)\neq 0$.
Let $Z_1,Z_2,\ldots$ be a sequence of i.i.d. r.v. distributed as $Z$, and set $N\coloneqq \inf\{n: Z_n\in A\}$.
Then we have the following identity in distribution,
\[
\left(Z_1,\ldots,Z_N\right) \,{\buildrel d \over =}\, \left(Z^\prime_1,\ldots,Z^\prime_{G},Z^{\prime\prime}_{G+1}\right) 
\]
where 
\begin{itemize}
\item $Z^\prime_1,\ldots Z^\prime_{G}$ are i.i.d. r.v. of probability distribution $\mathbf{P}\left(Z\in \cdot \middle\vert Z\notin A\right)$
\item $Z^{\prime\prime}_{G+1}$ is an independent r.v. distributed as $\mathbf{P}\left(Z\in \cdot \middle\vert Z\in A\right)$
\item $G$ is independent of $Z^\prime_1,\ldots,Z^\prime_{G},Z^{\prime\prime}_{G+1}$, and has geometric distribution with parameter $p$, i.e. $\mathbf{P}(G=k)=(1-p)^{k}p$, for $k\geq 0$.
\end{itemize} 
\end{lemma}

\section{Results}\label{sec:results}

From now on we consider a finite measure $\Pi$, on $(0,+\infty)$, of mass $b$ and $m\coloneqq\int_0^{+\infty}r\Pi(\der r)$ and $\Y$ a spectrally positive Lévy process with drift coefficient $d=-1$, Lévy measure $\Pi$, and Laplace exponent denoted by $\psi$, i.e.,
\begin{equation}
\psi(\lambda)=\lambda - \int\limits_0^\infty (1-\e^{-\lambda r})\Pi(\der r), \ \lambda\geq 0 \label{eq:Lap_exp}
\end{equation}

As in the preliminaries, we denote by $P_x$ the law of the process conditional on $\Y_0=x$, by $W$ the corresponding scale function and by $\eta$ the largest root of $\psi$. For any $s\in \R_+$, denote by $\Y^{(s)}$ the process reflected below $s$, as defined in the preliminaries.

We also recall the definition of the measure $\widetilde\Pi(\der y)\coloneqq \e^{-\eta y} \Pi(\der y)$ on $(0,+\infty)$. Then $\widetilde\psi, \widetilde m, \widetilde{W}, \widetilde{P}, \widetilde{Y}$, will denote the Laplace exponent, the mean value, the scale function, the law, and the process itself with Lévy measure $\widetilde\Pi$. 
As we have seen before, this spectrally positive Lévy process with Laplace exponent $\widetilde\psi$, killed when it hits 0, has the same law as $(Y_t,0\leq t<\tau_0)$ conditioned on hitting $0$, when starting from any $x>0$, that is $\widetilde P_x\circ k^{-1}_{\tau_0} = P_x(\cdot\vert \tau_0<+\infty)\circ k^{-1}_{\tau_0}$.
\newline

Fix $p\in(0,1)$ and $T\in(0,+\infty)$. 
We define a forest $\f^p$ consisting in $N_p+1$ splitting trees with lifespan measure $\Pi$ as follows,
\[\f^p \coloneqq\left(\ctree_1,\ldots,\ctree_{N_p},\ctree_{N_p+1}\right)\]
where,
\medskip
\begin{itemize}
\item $\ctree_1,\ldots\ctree_{N_p}$ are i.i.d. splitting trees conditioned on extinction before $T$
\item $\ctree_{N_p+1}$ is an independent splitting tree conditioned on survival up until time $T$
\item $N_p$ is a geometric random variable with parameter $p$, independent of all trees in the sequence, i.e. $\Pro(N_p=k)=(1-p)^kp$, for $k\geq 0$.
\end{itemize} 

Notice that, if $p=\Pro(\xi_T(\ctree)\neq 0)$, then $\f^p$ has the same distribution as a sequence of i.i.d. splitting trees, stopped at its first element surviving up until time $T$ (see Lemma~\ref{lemma:geom}). Hereafter we will frequently make use of this identity in law.
\medskip

We will add a subscript to denote equally constructed forests, but where the i.i.d. lifetimes of the ancestors on the splitting trees are different from that of their descendants. More precisely we will refer to $\f_\bot^p$ or $\f_\top^p$ if the ancestors are distributed respectively as the overshoot and undershoot defined by (\ref{eq:overshoot}) and (\ref{eq:undershoot}), conditional on $\{\tau_0^+<+\infty\}$, i.e.,
\[
\bot:\ \zeta(\varnothing) \sim \mu_\bot \qquad \qquad \qquad 
\top:\ \zeta(\varnothing) \sim \mu_\top
\]
We follow the same convention for splitting trees, that is, $\ctree_\top$ and $\ctree_\bot$ denote trees starting from one ancestor with these distributions, as well as for their probability laws, denoted by $\Pro_\top,\Pro_\bot$.

Finally, we use the notation $\widetilde{\f}^p$ when the lifespan measure of all individuals is $\widetilde\Pi$, instead of $\Pi$. The addition of a subscript is assumed to affect the lifetime distribution of the ancestors in the exact same way as described before.

\medskip

We start with the following result, which is the extension of Propositions~\ref{prop:duality} and \ref{theo:lambert-2010} to the case of these splitting trees with size-biased ancestors. Its proof is given later in the Appendix.
\begin{lemma}\label{lemma:contour-trees}
Let $\ctree_\bot$ be a splitting tree and $Y$ a spectrally positive Lévy process, as defined in Section~\ref{sec:results}, then the contour process $\Co=\Co(\ctree_\bot)$ has the following properties

\begin{enumerate}[label={\upshape(\roman*)}]
\item Under $\Pro_\bot\left(\cdot\middle\vert \textrm{Ext} \right)$, $\Co$ has the same distribution as $\left(\Y_{\tau_0^++t}, 0\leq t\leq \tau_0-\tau_0^+\right)$ under $P_0\left(\cdot\middle\vert \tau_0<+\infty\right)$.
\item Under $\Pro_\bot\left(\cdot\middle\vert \xi_T=0\right)$, $\Co$ has the distribution of $\left(\Y_{\tau_0^++t}, 0\leq t\leq \tau_0-\tau_0^+\right)$ under $P_0\left(\cdot\middle\vert \tau_0<\tau_T^+,\tau_0^+<+\infty\right)$.
\item Under $\Pro_\bot$, the contour of the truncated tree $\ctree_\bot^{(T)} $, is distributed as $\left(\Y_{\tau_0^++t}, 0\leq t\leq \tau_0-\tau_0^+\right)$ under $P_0\left(\cdot\middle\vert \tau_0^+<+\infty\right)$, reflected at $T$ and killed upon hitting 0.
\end{enumerate}
\end{lemma}
\medskip

Define now the two parameters,
\[
\gamma\coloneqq\frac{1}{W(T)} \qquad \qquad \textrm{and} \qquad \qquad \widetilde\gamma\coloneqq\frac{1}{\widetilde W(T)}.
\]
We have the following two results on forests,
\begin{lemma}\label{lemma:gamma-tilde-gamma}
\[
\widetilde\gamma = \Pro_\bot\left(\xi_T\neq 0\right) \qquad \textrm{and} \qquad \gamma = \widetilde\Pro_\top\left(\xi_T\neq 0\right)
\]
\end{lemma}
\begin{proof}\let\qed\relax
See Appendix.
\end{proof}

\begin{lemma}\label{lemma:f-etoile}
In the supercritical and critical cases ($m\geq1$) we have,

$\f^{\tilde\gamma}_\bot \ \,{\buildrel d \over =}\, \ \f^* \coloneqq$ a sequence of i.i.d. splitting trees with law $\Pro_\bot$ stopped at the first tree having survived up to time $T$.

$\widetilde\f^{\gamma}_\top\ \,{\buildrel d \over =}\, \ \widetilde\f^* \coloneqq$ a sequence of i.i.d. splitting trees with law $\widetilde\Pro_\top$ stopped at the first tree having survived up to time $T$.
\end{lemma}

\begin{proof}
By definition, a forest $\f^*$ consists in a number of trees, say $\widetilde N+1$, where $\widetilde N$ is a geometric random variable with probability of success $\Pro_\bot (\xi_T \neq 0)$, counting the trees that die out before $T$, until there is one that survives.
Hence, thanks to Lemma~\ref{lemma:geom}, the only thing that remains to prove is that $\widetilde\gamma$ is exactly this probability of success for the forest $\f^*$, in the same way that $\gamma$ for the forest $\widetilde\f^*$, which is the statement in Lemma~\ref{lemma:gamma-tilde-gamma}. 
\end{proof}

Then we are ready to state our first result concerning the population size processes of these forests,

\begin{theorem}\label{theo:main}
We have the following identity in distribution,
\[\left(\xi_{T-t}\left(\f^{\tilde\gamma}_\bot\right), 0\leq t\leq T\right)\,{\buildrel d \over =}\,\left(\xi_{t}\left(\widetilde\f^{\gamma}_\top\right), 0\leq t\leq T\right)\]
In the subcritical and critical cases ($m\leq 1$),
\[\left(\xi_{T-t}\left(\f^{\gamma}_\bot\right), 0\leq t\leq T\right)\,{\buildrel d \over =}\,\left(\xi_{t}\left(\f^{\gamma}_\top\right), 0\leq t\leq T\right)\]
and actually in this case $\bot=\top$ since in both cases, $\zeta(\varnothing)$ has density $\dfrac{\overline\Pi(r)}{m}\der r$.

In the supercritical and critical cases ($m\geq 1$) we have
\[\left(\xi_{T-t}\left(\f^*\right), 0\leq t\leq T\right)\,{\buildrel d \over =}\,\left(\xi_{t}\left(\widetilde\f^*\right), 0\leq t\leq T\right)\]
\end{theorem}

\begin{remark}
Theorem~\ref{theo:intro} from the Introduction is a particular case of this theorem when $\Pi$ is exponential.
\end{remark}

\begin{remark}\label{rem:subcrit-case}
When $m<1$, $\widetilde\f^{\gamma}_\bot$ has no interpretation as a stopped sequence of i.i.d. splitting trees as in Lemma~\ref{lemma:f-etoile}, because $\gamma\neq \widetilde\Pro_\top\left(\xi_T\neq 0\right)$.
Indeed, in this case
\begin{eqnarray}
&&P_0\left(\tau_0<\tau_T^+\middle\vert\tau_0^+<+\infty\right) 
= \int\limits_0^T
P_v\left(\tau_0<\tau^+_T\right) P_0\left(Y_{\tau_0^+}\in \der v\middle\vert\tau_0^+<+\infty\right)
\nonumber \\
&&\qquad = \int\limits_0^T \dfrac{W(T-v)}{W(T)} \dfrac{\overline{\Pi}(v)}{m} \der v = \dfrac{1}{m}\left(1-\dfrac{1}{ W(T)}\right) \nonumber
\end{eqnarray}
where the last equality comes from (\ref{eq:convol-W}). 
This entails that in this case the number of trees on the forest $\widetilde\f^*$ is geometric with parameter
\[
1 - \dfrac{1-\frac{1}{ W(T)}}{1-\frac{1}{W(\infty)}} \neq \gamma.
\]
\end{remark}

\bigskip
Actually, we will state a more general equality in distribution, concerning not only the underlying population size processes of the forests, but the two \textit{dual} forests themselves (see Fig.~\ref{fig:dual-forest}). 
For a forest $\f$ consisting of $N$ chronological trees that go extinct before $T$, and an $(N +1)$-st tree $\ctree_{N+1}$ that reaches time $T$, truncated up to this time, its \textit{dual} forest (in reverse time) can be defined as follows: its roots are the individuals of $\f$ extant at $T$, birth events become death events and vice versa, and the parental relations are re-drawn from the top of edges to the right (when looking in the original time direction), such that daughters are now to the left of their mothers. This rule is applied to all edges, except for those which are to the right of the last individual that survives up until time $T$, that are translated to the left of the first ancestor before re-drawing the parental relations, as it is shown in Fig.~\ref{fig:dual-forest}.
Hence, we postpone the proof of Theorem~\ref{theo:main}, that will be established later as a consequence of this more general result.

\begin{center}
\begin{figure}
\includegraphics[scale=1]{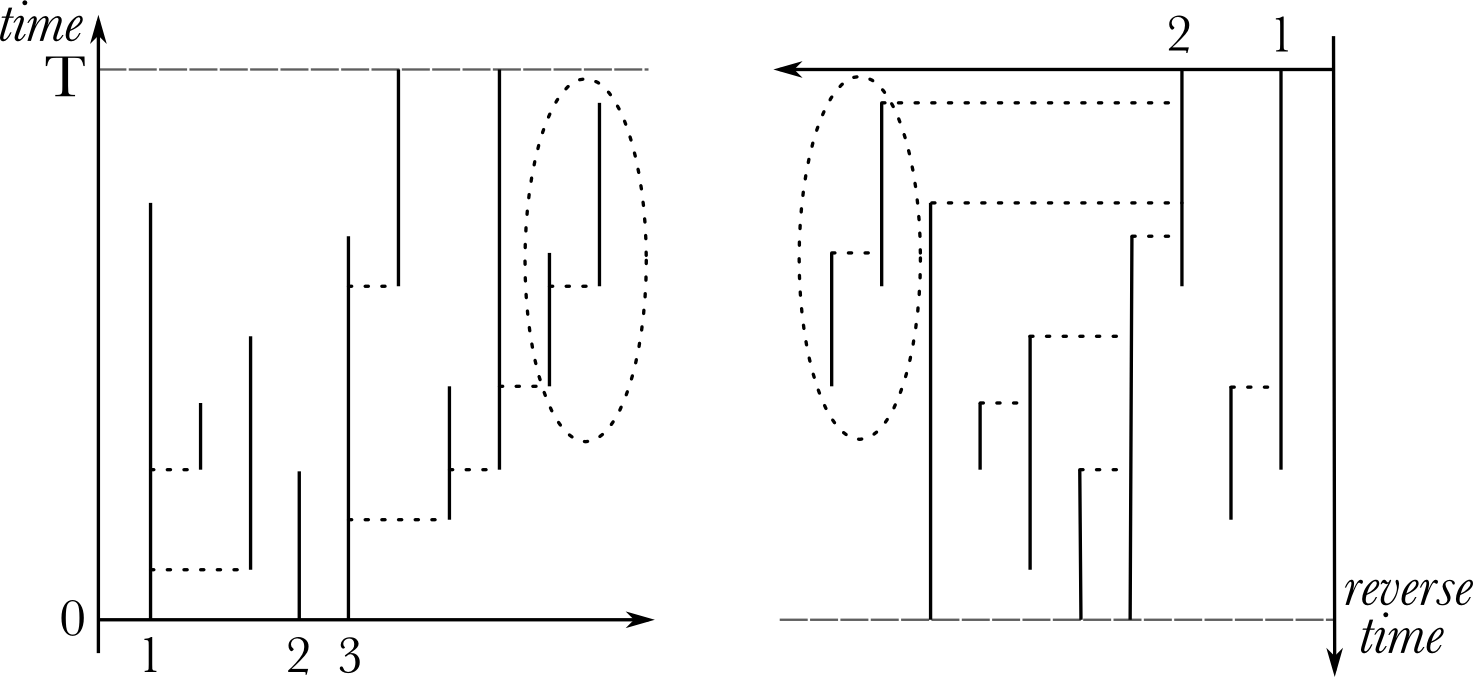}
\caption{An example of forest $\f^*$ consisting in three chronological trees (left) and its dual (right).}
\label{fig:dual-forest}
\end{figure}
\end{center}

Hereafter, we will consider the contour process of a forest $\f$ truncated at $T$ (i.e. each tree is truncated at $T$) for which we use the following notation $\Co^{(T)} (\f) = \Co(\f^{(T)})$.
Also define for a càdlàg process $X$, a deterministic transformation of its path by 
\[
\K(X)\coloneqq\rho\circ \chi(X)
\] 
which exists when $g_0(X)<+\infty$ (see preliminaries). This operator has the effect of shifting the last excursion from $T$ to 0 to the left and then apply the space-time-reversal of the process as defined in the preliminaries.

\begin{theorem}\label{theo:main1} 
Let $\Co^{(T)} (\f^{\tilde\gamma}_\bot) $ be the JCCP of a random forest $\f^{\tilde\gamma}_\bot$ truncated at $T$.
Then, after applying the operator $\K$, the process obtained has the law of the contour of a forest $\widetilde\f^{\gamma}_\top$, also truncated at $T$.
More precisely,
\[
\K\left( \Co^{(T)}  (\f^{\tilde\gamma}_\bot)\right) \,{\buildrel d \over =}\, \Co^{(T)}  \left(\widetilde\f^{\gamma}_\top\right)
\]
\end{theorem}
\medskip

Now the proof of Theorem~\ref{theo:main} can be achieved as a quite immediate consequence of this second theorem.
\begin{proof}[Proof of Theorem~\ref{theo:main}]
Recall our definition of the local time $\Gamma$ of a process $\Y$ with finite variation and finite lifetime. We only need to notice that, for any càdlàg function $X$ such that $g_0(X)<+\infty$,
\[
\left(\Gamma_{T-t}(X\circ k_{g_0(X)}),0\leq t\leq T \right)= \left(\Gamma_{t}\circ \K(X),0\leq t\leq T\right).
\]
This is true, in particular, when $X$ is the contour of a truncated forest $\Co^{(T)} (\widetilde\f_\top^\gamma)$, which lifetime is precisely $g_0(\Co)$.
Since the local time process of the contour of a tree, $\Gamma\circ \Co$, is the same as its population size process $\Xi$, the first result is established.

The second statement about the subcritical and critical cases is immediate from the first one, and the fact that measures $\mu_\top$ and $\mu_\bot$ are the same, as well as $\Pi$ and $\widetilde \Pi$, when $m\leq 1$.

Finally, the third identity is also a consequence of the first one and Lemma~\ref{lemma:f-etoile}.
\end{proof}

To demonstrate Theorem~\ref{theo:main1} we will consider first two independent sequences, with also independent elements, distributed as excursions of the process $Y$ starting at 0 or $T$ and killed upon hitting 0 or $T$. 
Notice that, $P_0$-a.s., we have $\{\tau_0<\tau_T^+,\tau_0^+<+\infty\} = \{\tau_0<\tau_T^+\}$ and $\{\tau^+_T<\tau_0,\tau_0^+<+\infty\} = \{\tau^+_T<\tau_0\}$, however we choose to use  the left-hand-side events in the definitions below, to emphasize the fact that the process is conditioned on hitting $(0,+\infty)$.
More precisely define the sequence $(\epsilon_i, 1\leq i\leq \widetilde N+1)$ as follows
\begin{itemize}
\item $(\epsilon_i)$ are i.i.d. with law $P_0(\cdot\vert \tau_0<\tau_T^+,\tau_0^+<+\infty)\circ k_{\tau_0}^{-1}$ for $1\leq i\leq \widetilde N$, 
\item $\epsilon_{\widetilde N+1}$ has law $P_0 (\cdot\vert \tau^+_T<\tau_0,\tau_0^+<+\infty)\circ k_{\tau_T^+}^{-1}$
\item $\widetilde N$ is an independent geometric random variable with probability of success $\widetilde\gamma= P_0 ( \tau^+_T<\tau_0,\vert\tau_0^+<+\infty)$.
\end{itemize}

Also define the sequence $(\widetilde{\epsilon}_i, 1\leq i\leq N+1)$ as
\begin{itemize}
\item $\widetilde\epsilon_i$ are i.i.d. with the law $P_T (\cdot\vert \tau^+_T<\tau_0)\circ k_{\tau_T^+}^{-1}$ for $1\leq i\leq N$
\item $\widetilde\epsilon_{N +1}$ has law $P_T(\cdot\vert \tau_0<\tau^+_T)\circ k_{\tau_0}^{-1}$
\item $N$ is an independent geometric random variable with probability of success $\gamma = P_T(\tau_0<\tau_T^+)$.
\end{itemize}

We denote by $Z$ the process obtained by the concatenation of these two sequences of excursions in the same order they were defined, that is $Z\coloneqq [\epsilon,\widetilde\epsilon]$ (see Fig.~\ref{fig:Z}).
We will prove first that, after a time change erasing the negative values of these excursions and closing up the gaps, the process thus obtained has the same law as the contour of the forest $\f^{\tilde\gamma}_\bot$ truncated at $T$.

\begin{figure}
\centering
\includegraphics[scale=0.47]{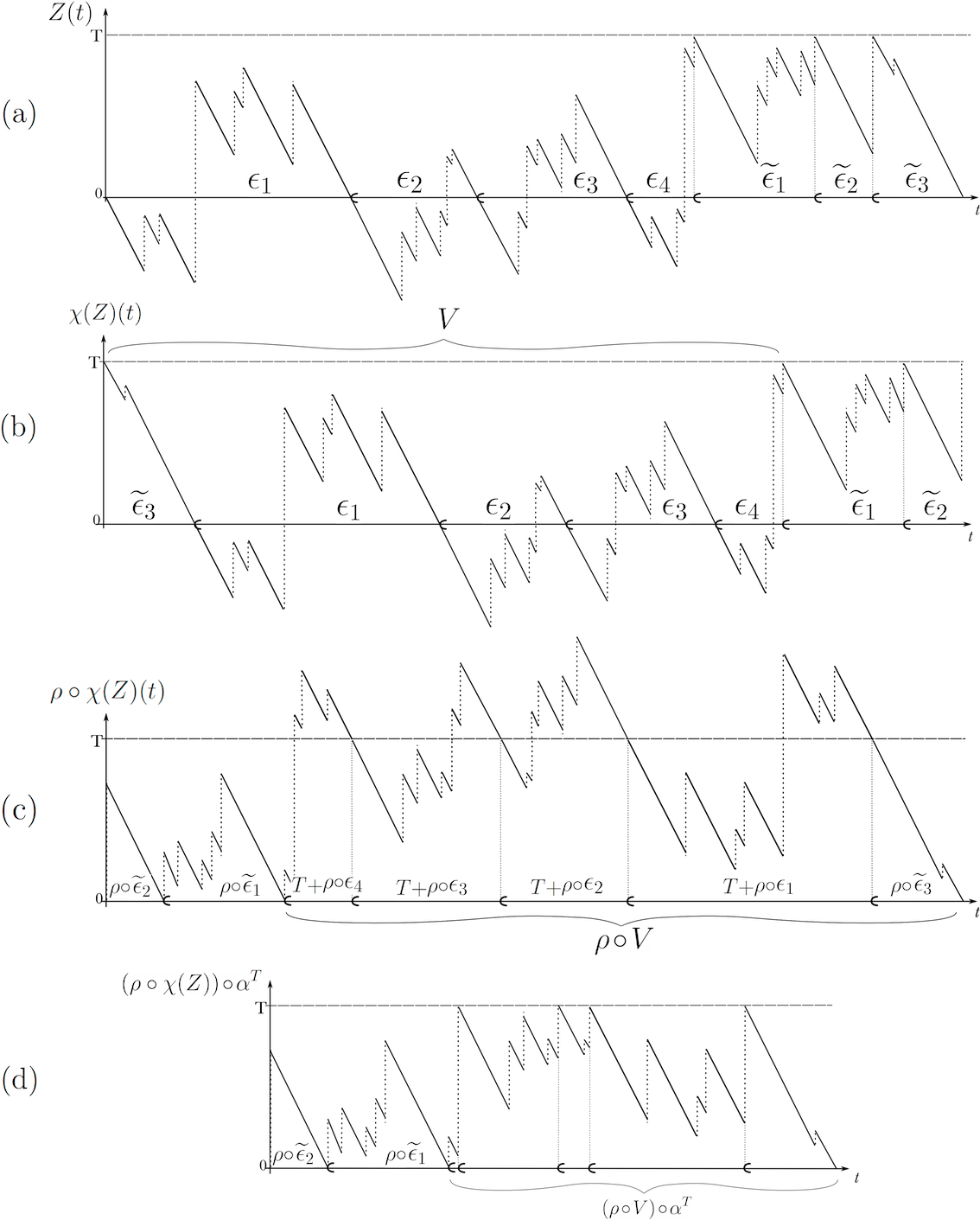}
\captionsetup{width=14cm}
\caption{(a) An example of the Lévy process $Z$ with $\widetilde N = 3$ excursions before hitting $T$, and $N=2$ excursions before hitting 0 again, and some path transformations of $Z$: (b) $\chi(Z)$ places the last excursion $\widetilde \epsilon_3$ at the origin and shifts to the right the rest of the path; (c) $\rho\circ \chi (Z)$ is the space-time-reversal of the path at (b); and (d) $\left(\rho\circ\chi(Z)\right)\circ\alpha^T $ erases the sub-paths of (c) taking values greater than $T$ and closes up the gaps, hence the shorter length of the path.}
\label{fig:Z}
\end{figure}

\begin{claim}\label{claim:zeta}
We have the following identity in law:
\[
 Z\circ\alpha \,{\buildrel d \over =}\,  \Co^{(T)} (\f^{\tilde\gamma}_\bot)
\]
\end{claim}

\begin{proof}
Since a forest $\f^{\tilde\gamma}_\bot$ is a finite sequence of independent trees and, when truncated, its contour process, say $\Co^{(T)} (\f^{\tilde\gamma}_\bot)\coloneqq \left(\Co_t, 0 \leq t\leq \overline\ell\right)$, is defined as the concatenation of the contour of the trees of this sequence, its law will be characterized by the law of the sequence of the killed paths $e_i\coloneqq\left(\Co_{t_i+t},0\leq t<t_{i+1}-t_i\right)$, where $t_0=0$, $t_i = \sum_{j=1}^i \ell(\ctree_{\bot,i})$ for $1\leq i\leq N_{\tilde\gamma}$,  $t_{N_{\tilde\gamma}+1}= t_{N_{\tilde\gamma}} + \ell(\ctree^{(T)}_{\bot,N_{\tilde\gamma}+1}) = \ell(\f^{\tilde\gamma}_\bot)$. Here $N_{\tilde\gamma}+1$ is the number of trees in $\f^{\tilde\gamma}_\bot$.
These killed paths are then the JCCP's of each of the trees in the forest, that is $e_i = \Co(\ctree_{\bot,i}^{(T)})$ for every $1\leq i\leq N_{\tilde\gamma}+1$, which are by definition independent, and their number $N_{\tilde\gamma}+1$ is geometric with parameter $\widetilde\gamma$. 
The first $N_{\tilde\gamma}$ are identically distributed and conditioned on extinction before $T$, and the last one conditioned on surviving up until time $T$. 

On the other hand, define for the sequence of excursions $\epsilon$, the lifetime of its terms, $\zeta_i = \zeta(\epsilon_i)$ and the first time they hit $[0,+\infty)$, $\zeta_i^+=\zeta^+(\epsilon_i)$. These excursions start at 0  and always visit first $(-\infty,0)$. Under $\{\tau_0^+<+\infty\}$, 0 is recurrent for the reflected process, so we have $0<\zeta_i^+<\zeta_i<+\infty$, $P_0$-a.s. for every $i\geq 1$.
If now we apply to each of these excursions the time change $\alpha$, removing the non positive values and closing up the gaps, we obtain for $1\leq i\leq \widetilde N+1$, 
\[
\epsilon_i\circ \alpha = \left(\epsilon_i(\zeta_i^++t),0\leq t< \zeta_i-\zeta_i^+ \right) 
\]
so, for $1\leq i\leq \widetilde N$, $\epsilon_i\circ \alpha$ has  the law of $(\Y_{\tau_0^++t}, 0\leq t\leq \tau_0-\tau_0^+)$ under $P_0(\cdot\vert \tau_0<\tau_T^+,\tau_0^+<+\infty)$.
Thanks to Lemma~\ref{lemma:contour-trees} (ii), this is the same as the law of the contour of each of the first $N_{\tilde\gamma}$ trees on the forest $\f^{\tilde\gamma}_\bot$.

The last excursion, $\epsilon_{\widetilde N+1}\circ \alpha$ has the law of $(\Y_{\tau_0^++t}, 0\leq t\leq \tau_T^+-\tau_0^+)$ under $P_0(\cdot\vert \tau_T^+<\tau_0,\tau_0^+<+\infty)$, that is, an excursion of $Y$ starting from an initial value distributed according to $\mu_\bot$, conditioned on hitting $T$ before 0. 
 
If we look now at the second sequence $\widetilde{\epsilon}$, we notice that, since $\gamma = P_T(\tau_0<\tau_T^+)$, it is distributed as a sequence of i.i.d. excursions with law $P_T\circ k_{\tau_0\wedge\tau_T^+} (\cdot)$, stopped at the first one hitting 0 before $(T,+\infty)$.
Thus, Theorem 4.3 in \cite{Lam10} 
guarantees that the concatenation $[\epsilon_{\widetilde N+1}\circ \alpha,\widetilde{\epsilon}]$ has the law of the contour of the last tree in the forest $\f^{\tilde\gamma}_\bot$ truncated at $T$, say $\ctree^{(T)}_{\bot,N_{\tilde\gamma}+1}$.

Finally, since $\widetilde N$ has the same distribution as $N_{\tilde\gamma}$, we have,
\[
[e_i, 1\leq i\leq N_{\tilde\gamma}+1] \,{\buildrel d \over =}\, [\epsilon_i\circ \alpha, 1\leq i\leq \widetilde N+1,\widetilde{\epsilon}_i, 1\leq i\leq N+1]
\]
Now the right-hand-side equals $[\epsilon,\widetilde\epsilon]\circ \alpha$, because the time change $\alpha$ is the inverse of an additive functional, so it commutes with the concatenation, and the elements of the sequence $\widetilde\epsilon$ do not take negative values, so the time change $\alpha$ has no effect on them.
This ends the proof of the claim.
\end{proof}

Now we will look at the process $Z$ after relocating the last excursion of the second sequence, $\widetilde\epsilon_{N +1}$ to the beginning and shifting the rest of the path to the right. More precisely, consider the process 
\[V=[\widetilde\epsilon_{N +1}, \epsilon_1,\ldots, \epsilon_{\widetilde N +1}],
\] 
and consider also the space-time-reversed process $\rho\circ V$ (see Fig.~\ref{fig:Z}). It is not hard to see that
\[
\rho\circ V = [T+\rho\circ \epsilon_{\widetilde N +1},T+\rho\circ \epsilon_{\widetilde N},\ldots,T+\rho\circ\epsilon_1,\rho\circ \widetilde\epsilon_{N +1} ],
\]
since $V(0)=\epsilon_{\widetilde N +1}(0) = T$ and all the other excursions in $V$ take the value 0 at 0. 
Then we have the following result on the law of this process, reflected at $T$.
\medskip

\begin{claim}\label{claim:rho-V-alpha}
We have the following identity in law:
\[
(\rho\circ V)\circ\alpha^T \,{\buildrel d \over =}\, \Co^{(T)}\left(\widetilde \ctree_{\top}^\prime\right)
\] 
where $\widetilde \ctree_{\top}^\prime$ is a splitting tree conditioned on surviving up util time $T$.
\end{claim}

For the proof of Claim~\ref{claim:rho-V-alpha} we will need the following two results that are proved in the Appendix.

\begin{lemma}\label{lemma:last-exc}
The probability measure $P_T\left(\cdot\middle\vert \tau_0<\tau_T^+\right)\circ k_{\tau_0}^{-1}$ is invariant under space-time-reversal.
\end{lemma}

\begin{lemma}\label{lemma:funct-law}
For any $a>0$, $x\in[0,a]$ and $\Lambda \in \mathscr{F}_{\tau_a}$ the following identity holds
\[
\widetilde P_x\left(\Lambda,\tau_a^+<\tau_0\right) =
P_x\left(\Lambda, \tau_a^+<\tau_0,\tau_a<+\infty\right) \e^{-\eta(a-x)}
\]
In particular, $\widetilde P_x\left(\tau_a^+<\tau_0\right) =P_x\left( \tau_a^+<\tau_0,\tau_a<+\infty\right) \e^{-\eta(a-x)}$, hence
\[
\widetilde P_x\left(\Lambda\vert \tau_a^+<\tau_0\right) =
P_x\left(\Lambda\vert \tau_a^+<\tau_0,\tau_a<+\infty\right) 
\]
\end{lemma}

\begin{proof}[Proof of Claim~\ref{claim:rho-V-alpha}]
We can deduce from Proposition~\ref{prop:duality} the following observations about the laws of the space-time-reversed excursions:
\begin{enumerate}
\item[(1)] conditional on $\epsilon_{\widetilde N +1}(\zeta-)=T-u$, the reversed excursion $T+\rho\circ \epsilon_{\widetilde N +1}$ has law $P_u(\cdot\vert\tau_T^+<\tau_0,\tau_T<+\infty)\circ k_{\tau_T}^{-1} $
\item[(2)] for $1\leq i\leq \widetilde N$, the excursions $T+\rho\circ\epsilon_i$ have law $P_T(\cdot\vert \tau_T^+<\tau_0,\tau_T<+\infty)\circ k_{\tau_T}^{-1}$
\end{enumerate}
and thanks to Lemma~\ref{lemma:last-exc}, we also have
\begin{enumerate}
\item[(3)] $\rho\circ \widetilde\epsilon_{ N +1}\,{\buildrel d \over =}\,\widetilde\epsilon_{ N +1}$, with common law $P_T(\cdot\vert \tau_0<\tau_T^+)\circ k_{\tau_0}^{-1}$.
\end{enumerate}

Now we would like to express the laws of the excursions in (1), (2) and (3) in terms of the probability measure $\widetilde P$, the probability of $\Y$ conditioned on not drifting to $+\infty$ (see preliminaries). For (3) we easily have, from (\ref{eq:prob-cond-hit-0}), that
\[
P_T(\cdot\vert \tau_0<\tau_T^+)\circ k_{\tau_0}^{-1} = P_T(\cdot\vert \tau_0<\tau_T^+,\tau_0<+\infty)\circ k_{\tau_0}^{-1} = \widetilde P_T (\cdot\vert \tau_0<\tau_T^+)\circ k_{\tau_0}^{-1}
\]

The result in Lemma~\ref{lemma:funct-law} entails in particular that excursions in (1) and (2), killed at the time they hit $(T,+\infty)$, have distribution $\widetilde P_u (\cdot\vert\tau_T^+<\tau_0)\circ k_{\tau_T^+}^{-1}$ and $\widetilde P_T(\cdot\vert \tau_T^+<\tau_0)\circ k_{\tau_T^+}^{-1}$ respectively. Besides, notice that to kill these excursions at $\tau_T^+$ is the same as applying the time change $\alpha^T$, i.e. removing the part of the path taking values in $(T,+\infty)$.

Then, conditional on $V(\zeta-)=T-u$, the reversed process after the time change $\alpha^T$, that is $(\rho\circ V)\circ\alpha^T$ consists in a sequence of independent excursions distributed as the Lévy process $\widetilde\Y$ killed at $\tau_0\wedge\tau_T^+$, all starting at $T$, but the first one, starting at $u$. There are $\widetilde N$ excursions from $T$, conditioned on hitting $T$ before 0 and a last excursion from $T$ conditioned on hitting 0 before returning to $T$, and killed upon hitting 0.
Observe that $\widetilde N$ has geometric distribution with parameter $\widetilde\gamma$, which is exactly the probability that an excursion of $\widetilde\Y$, starting at $T$, exits the interval $(0,T)$ from the bottom, that is $\widetilde P_T(\tau_0<\tau_T^+)$ (equation~\ref{eq:exit-two-side}). This implies that excursions in (2) and (3) after the time change $\alpha^T$, form a sequence of i.i.d. excursions of $\widetilde \Y$ starting at $T$, killed at $\tau_0\wedge\tau_T^+$, ending at the first one that hits 0 before $[T,+\infty)$ (Lemma~\ref{lemma:geom}).

The fact that the time change $\alpha^T$ commutes with the concatenation, allows to conclude that $(\rho\circ V)\circ\alpha^T$, conditional on $V(\zeta-)=T-u$, has the law of the process $\widetilde\Y$ starting at $u$, conditioned on hitting $T$ before 0, reflected below $T$ and killed upon hitting 0. 

From the definition of the sequence $\epsilon$ we deduce that $V(\zeta-) = \epsilon_{\widetilde N +1}(\zeta-)$ has the law of the undershoot of $Y$ at $T$ of an excursion starting at $T$ and conditioned on hitting $0$ before $(T,+\infty)$. As usual, the strong Markov property and the stationary increments of the Lévy process entail that this excursion is invariant under translation of the space, hence this undershoot has the distribution 
$P_0\left(-\Y_{\tau_0^+-}\in\cdot\middle\vert\tau_{-T}<\tau_0^+<+\infty \right)$. This implies that the law of $(\rho\circ V)\circ\alpha^T$ is
\begin{equation}\label{eq:undershoot-cond}
\int\limits_{0}^{T} P_0\left(-\Y_{\tau_0^+-}\in\der u\middle\vert\tau_{-T}<\tau_0^+<+\infty \right) \widetilde P_u\left(\Y^{(T)}\in\cdot\middle\vert\tau_{T}^+<\tau_0 \right)\circ k_{\tau_0}^{-1},
\end{equation}
and we will show this is the same as
\begin{equation}\label{eq:undershoot-cond1}
\widetilde P_\top \left( \cdot\middle\vert\tau_T^+<\tau_0\right)\circ k_{\tau_0}^{-1}
= \frac{\int_0^T\mu_\top(\der u) \widetilde P_u\left(\Y^{(T)}\in\cdot,\tau_{T}^+<\tau_0 \right)\circ k_{\tau_0}^{-1}}{\int_0^T\mu_\top(\der u) \widetilde P_u\left(\tau_{T}^+<\tau_0 \right)}.
\end{equation}
The strong Markov property and Proposition~\ref{prop:duality} imply that
\begin{eqnarray}
P_0\left(-\Y_{\tau_0^+-}\in\der u\middle\vert\tau_{-T}<\tau_0^+<+\infty \right)  =
\mu_\top(\der u)
\dfrac{P_u\left(\tau_{T}^+<\tau_0\middle\vert\tau_0<+\infty\right)}{P_0\left(\tau_{-T}<\tau_0^+\middle\vert\tau_0^+<+\infty\right)}. \nonumber
\end{eqnarray}
Then by using this identity and (\ref{eq:prob-cond-hit-0}), we have that (\ref{eq:undershoot-cond}) equals
\begin{eqnarray}
&&\int\limits_{0}^{T} \mu_\top(\der u)
\dfrac{P_u\left(\tau_{T}^+<\tau_0\middle\vert\tau_0<+\infty\right)}{P_0\left(\tau_{-T}<\tau_0^+\middle\vert\tau_0^+<+\infty\right)}
\dfrac{\widetilde P_0\left(\Y^{(T)}\in\cdot,\tau_{T}^+<\tau_0 \right)\circ k_{\tau_0}^{-1}}{\widetilde P_u(\tau_T^+<\tau_0)} \nonumber \\
&&=\dfrac{\int\limits_{0}^{T} \mu_\top(\der u) \widetilde P_0\left(\Y^{(T)}\in\cdot,\tau_{T}^+<\tau_0 \right)\circ k_{\tau_0}^{-1}}{P_0\left(\tau_{-T}<\tau_0^+\middle\vert\tau_0^+<+\infty\right)}. \nonumber
\end{eqnarray}
Finally, the numerator in the last term is the same as the one in the right-hand-side of (\ref{eq:undershoot-cond1}), due again to (\ref{eq:prob-cond-hit-0}) and the dualities from Proposition~\ref{prop:duality}.

Then, as announced, $(\rho\circ V)\circ\alpha^T$ has the law \ref{eq:undershoot-cond1}, which is, thanks to Lemma~\ref{lemma:contour-trees}, the contour of a splitting tree of lifespan measure $\widetilde\Pi$ and ancestor distributed as $\mu_\top$, conditioned on surviving up until time $T$ and truncated at $T$, say $\widetilde \ctree_{\top}^\prime$.
\end{proof}

It remains to understand the effect on the i.i.d. excursions $(\widetilde{\epsilon}_i, 1\leq i\leq N)$, of the reversal operator $\rho$, which is given in the following statement.

\begin{claim}\label{claim:rho-epsilon-tilde}
For $1\leq i\leq N$,
\[
\rho\circ\widetilde{\epsilon}_i \,{\buildrel d \over =}\, \Co\left( \widetilde \ctree_{\top,i}\right)
\]
where $\widetilde \ctree_{\top,i}$ is a sequence of i.i.d. splitting trees conditioned on dying out before $T$.
\end{claim}
\begin{proof}
We know from Proposition~\ref{prop:duality} that, conditional on $\widetilde{\epsilon}_i(\zeta-)=T-u$, the reversed excursions $\rho\circ \widetilde{\epsilon}_i$, has the law of $\widetilde\Y$ starting from $u$, conditioned on hitting 0 before $T$ and killed upon hitting 0, that is $\widetilde P_u(\cdot\vert\tau_0<\tau^+_T)\circ k_{\tau_0}^{-1}$.

The same reasoning as for Claim~\ref{claim:rho-V-alpha} yields that $T-\widetilde{\epsilon}_i(\zeta-)$ has the distribution $P_0\left(-\Y_{\tau_0^+-}\in\cdot\middle\vert\tau_0^+<\tau_{-T}<+\infty \right) $ for every $1\leq i\leq N$, and then, each of the reversed excursions $\rho\circ\widetilde{\epsilon}_i$ has the law $\widetilde P_\top(\cdot\vert \tau_0<\tau_T^+)\circ k_{\tau_0}^{-1}$. This is the same as the contour process of splitting trees, say $\widetilde \ctree_{\top,i}$, all i.i.d. for $1\leq i\leq N$ and conditioned on dying out before $T$.
\end{proof}

\begin{proof}[Proof of Theorem~\ref{theo:main1}]
We have now all the elements to complete the proof of this result.
Notice the transformations we have done to the trajectories of the process $Z=[\epsilon,\widetilde\epsilon]$ in terms of its excursions, can also be expressed in terms of the time-changes $\alpha, \alpha^T$ and the path transformations $\rho,\chi$ (Fig.~\ref{fig:Z}), as follows
\[
\chi(Z)= [V,\widetilde \epsilon_1,\ldots,\widetilde \epsilon_N] 
\]
and stressing that all these paths start by taking the value $T$, we have
\[
\rho\circ \chi( Z)	= [\rho\circ \widetilde\epsilon_N,\ldots,\rho\circ \widetilde \epsilon_1,\rho\circ V].
\]
On the other hand, after Claim~\ref{claim:zeta}, we have $\Co^{(T)}(\f_\bot^{\tilde\gamma})= Z\circ\alpha$, so
\[
\K(\Co^{(T)}(\f_\bot^{\tilde\gamma})) \,{\buildrel d \over =}\, \K\circ Z\circ\alpha = (\rho\circ \chi)\circ (Z\circ \alpha) =  
\]
\[
=[\rho\circ \widetilde\epsilon_N,\ldots,\rho\circ \widetilde \epsilon_1,(\rho\circ V)\circ\alpha^T] 
\]
We have proved that the right-hand term in the last equation has the law of the contour of a sequence of independent splitting trees $\widetilde \ctree_{\top,i}$, which are i.i.d. for  $1\leq i\leq N$, conditioned on dying out before $T$, and a last tree $\widetilde \ctree_{\top}$ conditioned on surviving up until time $T$. Since $N$ is geometric with parameter $\gamma$, this is the same as the process $\Co^{(T)}(\widetilde\f^\gamma_\top)$, which concludes the proof.
\end{proof}

\section{Epidemiology}\label{sec:epi}

In general, in the context of  epidemiology, phylogenetic trees are considered to be estimated from genetic sequences obtained at a single time point, or sampled sequentially through time since the beginning of the epidemic. 
There exists several methods allowing this estimation, which are not addressed here.
We assume the estimated \textit{reconstructed trees} (i.e. the information about non sampled hosts is erased from the original process) are the transmission trees from sampled individuals and no uncertainty on the branch lengths is considered.
This hypothesis makes sense when the epidemiological and evolutionary timescales can be supposed to be similar \cite{Vol13}.
These reconstructed phylogenies can provide information on the underlying population dynamic process \cite{Tho75,NMH94,DruPy03} and there is an increasing amount of work on this relatively new field of phylodynamics.

Most of phylodynamic models are based on Kingman's coalescent, but its poor realism in the context of epidemics (rapid growth, rapid fluctuations, dense sampling) has motivated other authors to use birth-death or SIR processes for the dynamics of the epidemics \cite{Vol09,Ra11,StaKo12,LaTr13,LaAlSt14}.
A common feature for most of these works is that they use likelihood-based methods that intend to infer the model parameters on the basis of available data, via maximum likelihood estimation (MLE), or in a Bayesian framework. 

Usually, not only the reconstructed phylogenies described above are available, but also incidence time-series, that is, the number of new cases registered through time (typically daily, weekly or monthly). This information may come from hospital records, surveillance programs (local or national), and is not necessarily collected at regular intervals.
As we mentioned before, here we are interested in the scenario where both types of data are available: phylogenetic trees (reconstructed from pathogen sequences) and incidence time series.

From the probabilistic point of view, we are interested in the distribution of the size process of the host population, denoted by $I\coloneqq(I_t, 0\leq t\leq T)$, conditional on the \textit{reconstructed transmission tree} from infected individuals at time $T$. 
More precisely, we want to characterize the law of $I$, conditional on $\sigma=(\sigma_i, 1\leq i\leq n)$ to be the coalescence times in the reconstructed tree of transmission from extant hosts at $T$.

We suppose that the host population has the structure of a forest $\f^{\tilde\gamma}_\bot$, so we consider there are a geometric number $\widetilde N$ (with parameter $\widetilde\gamma$) of infected individuals at time 0, for which the corresponding transmission tree dies out before $T$, and a last one which is at the origin of all the present-time infectives. Let $(H_i)_{1\leq i\leq N}$ be the coalescence times between these infected individuals at $T$, where as before, $N$ is an independent geometric r.v. with parameter $\gamma$. 
Here we are interested in characterizing the distribution of $\Xi(\f_\bot^{\tilde\gamma})$, the total population size process of infected individuals on $[0,T]$, conditional on $H_i=\sigma_i, 1\leq i\leq N$.

Since in our model, all infected individuals at $T$ belong to the last tree in the forest, which we know from the definition, is conditioned on surviving up until time $T$, a result from \cite{Lam10} and the pathwise decomposition of the contour process of a forest from the previous section, tells us that these coalescence times are precisely the depths of the excursions away from $T$ of  $\Co(\f_\bot^{\tilde\gamma})$, that is, $H_i \overset{d}{=} \inf \widetilde \epsilon_i$ for $1\leq i\leq N$. We recall that $\widetilde\epsilon_i$ are i.i.d. with law $P_T (\cdot\vert \tau^+_T<\tau_0)\circ k_{\tau_T^+}^{-1}$.

Then, conditioning the population size process on the coalescence times is the same as these excursions of the Lévy process $\Y$ conditioned on their infimum, which becomes, after the corresponding space-time-reversal, their supremum, since, $P_T(\cdot\vert \tau_T^+<\tau_0) \circ k^{-1}_{\tau_T^+}$-a.s., we have $\sup_{[0,\tau_T^+)}\rho\circ\widetilde\epsilon = \sup_{t\in[0,\tau_T^+)}T-\widetilde\epsilon((\tau_T^+-t)-) = \inf_{[0,\tau_T^+)}\widetilde\epsilon$. 
Besides, thanks to Theorem~\ref{theo:main1}, when we reverse the time, these excursions themselves are distributed as the contour of independent subcritical (with measure $\widetilde\Pi$) splitting trees conditioned on hitting 0 before $T$, starting  from a value distributed as $\mu_\top$ in $[0,T]$. 
Therefore, conditioning these excursions on their supremum is the same as conditioning the corresponding trees on their height, that is, conditioning on the time of extinction of each tree with lifespan measure $\widetilde\Pi$ to equal the corresponding time of coalescence $\sigma_i$. We notice that conditioning on an event as $\{T_{\textrm{Ext}}=s\}$, is possible here since the time of extinction of a tree $\widetilde\ctree_\top$ always has a density (because $\mu_\top$ has a density).

We also know from the proof of Theorem~\ref{theo:main1}, that the total population process of the forest, also takes into account the width process of the excursion $V\circ \alpha$, which is independent of $\widetilde\epsilon_i,1\leq i\leq N$, and when reversed, has the law of the contour of a splitting tree truncated up to $T$, and conditioned on surviving up until time $T$, that is $(\rho\circ V)\circ \alpha^T \overset{d}{=} \Co^{(T)}(\widetilde\ctree_\top ^\prime)$ as stated in Claim~\ref{claim:rho-V-alpha}.

More precisely we have the following result,

\begin{theorem}
Let $\f^{\tilde\gamma}_\bot$ and $(H_i)_{i\geq 1}$ as defined before. 
Then, under $P(\cdot\vert H_i=\sigma_i,1\leq i\leq N)$, the population size process backward in time, $(\xi_{T-t}(\f^{\tilde\gamma}_\bot), 0\leq t\leq T)$, is the sum of the width processes of $N+1$ independent splitting trees $(\widetilde \ctree_{\top,i})_{1\leq i\leq N+1}$, where,
\begin{itemize}
\item for $1\leq i\leq N$, $\widetilde \ctree_{\top,i}$ are subcritical splitting trees with lifespan measure $\widetilde \Pi$, starting with an ancestor with lifespan distributed as $\mu_\top$, and conditioned on its time of extinction to be $\sigma_i$, that is with law $\widetilde \Pro_\top(\cdot\vert T_{\textrm{Ext}}=\sigma_i)$
\item the last tree $\widetilde \ctree_{\top,N+1}$ is a subcritical splitting tree with lifespan measure $\widetilde \Pi$, starting with an ancestor with lifespan distributed as $\mu_\top$, and conditioned on surviving up until time $T$, that is with law $\widetilde \Pro_\top(\cdot\vert T_{\textrm{Ext}}> T)$.
\end{itemize}
\end{theorem}




\appendix

\section{Remaining proofs}

In this section we proceed to the proofs of Lemma~\ref{lemma:contour-trees}, Lemma~\ref{lemma:gamma-tilde-gamma}, Lemma~\ref{lemma:last-exc} and Lemma~\ref{lemma:funct-law}.

\begin{proof}[Proof of Lemma~\ref{lemma:contour-trees}]
These statements are quite immediate consequences of Theorem 4.3 from \cite{Lam10}, the strong Markov property and stationarity of the increments of $Y$.
We will expand the arguments for (i) and the other statements can be proved similarly.
We know from \ref{theo:lambert-2010}, that conditional on $\textrm{Ext}$ and $\zeta(\varnothing)=x$, the contour  $\Co=\Co(\ctree_\bot)$ has the law of $\left(Y_t,0\leq t\leq \tau_0\right)$ under $P_x\left(\cdot\middle\vert\tau_0<+\infty\right)$. Hence, it follows from the definition of $\ctree_\bot$ and by conditioning on its ancestor lifespan that,
\begin{eqnarray}
&& \Pro_\bot\left(\Co\in \cdot \middle\vert \textrm{Ext}\right) = \int\limits_0^\infty \Pro\left( \Co\in \cdot \middle\vert \textrm{Ext}, \zeta(\varnothing)=x\right) \Pro_\bot\left(\zeta(\varnothing)\in \der x\right) \nonumber \\
&& \qquad = \int\limits_0^\infty P_x\left( Y\circ k_{\tau_0}\in \cdot \middle\vert \tau_0<+\infty\right) P_0\left(Y_{\tau_0^+}\in \der x\middle\vert \tau_0^+<+\infty\right) \nonumber \\
&& \qquad = \int\limits_0^\infty P_0\left( Y\circ\theta_{\tau_0^+}\circ k_{\tau_0}\in \cdot \middle\vert Y_{\tau_0^+} = x, \tau_0<+\infty\right) P_0\left(Y_{\tau_0^+}\in \der x\middle\vert \tau_0^+<+\infty\right) \nonumber  \\
&& \qquad = P_0\left( Y\circ\theta_{\tau_0^+}\circ k_{\tau_0}\in \cdot \middle\vert \tau_0<+\infty\right), \nonumber 
\end{eqnarray}
which ends the proof of (i).
\end{proof}

\begin{proof}[Proof of Lemma~\ref{lemma:gamma-tilde-gamma}]

First we want to prove that $\widetilde{\gamma} =\Pro_\bot\left(\xi_T\neq 0\right)$. 
To do that, we can express this probability in terms of the contour process, thanks to Lemma~\ref{lemma:contour-trees} (ii), we have
\begin{eqnarray}
&&\Pro_\bot \left(\xi_T= 0\right) = P_0\left( \tau_0<\tau_T^+\middle \vert \tau_0^+<+\infty\right) \nonumber
\end{eqnarray}
According to (\ref{eq:return-to-zero-finite}), in the supercritical case, we have $P_0(\tau_0^+<+\infty)=1$.
The probability that the process $Y$, starting from $0$, returns to $0$ before it reaches the interval $[T,+\infty)$, can be computed by integrating with respect to the  measure of the overshoot at $0$, of an excursion starting from $0$. 
Then, by applying the strong Markov property, equations (\ref{eq:exit-two-side}) and (\ref{eq:convol-W-tilde}), we prove the first identity,
\begin{eqnarray}
&&P_0\left(\tau_0<\tau_T^+\right) 
= \int\limits_0^T
P_v\left(\tau_0<\tau^+_T\right) P_0\left(Y_{\tau_0^+}\in \der v\right)
\nonumber \\
&&\qquad = \int\limits_0^T P_v \left( \tau_0<\tau_T^+\right) \e^{\eta v} \overline{\widetilde\Pi}(v) \der v  =  \int\limits_0^T \dfrac{W(T-v)}{W(T)} \e^{\eta v} \overline{\widetilde\Pi}(v) \der v \nonumber \\
&&\qquad =  \dfrac{1}{W(T)\e^{-\eta T}}\int\limits_0^T W(T-v)\e^{-\eta (T- v)} \overline{\widetilde\Pi}(v) \der v  =  \dfrac{1}{\widetilde W(T)}\int\limits_0^T \widetilde W(T-v) \overline{\widetilde\Pi}(v) \der v  \nonumber \\
&&\qquad = 1 - \dfrac{1}{\widetilde W(T)} \nonumber 
\end{eqnarray}
\smallskip

The second statement is that $\gamma =  \widetilde\Pro_\top\left(\xi_T\neq 0\right)$, which follows from,
\begin{eqnarray}
&&\widetilde P_\top\left(\tau_0<\tau_T^+\right) =\int\limits_0^T  \widetilde P_u\left(\tau_0<\tau_T^+\right) P_0\left(-Y_{\tau_0^+-}\in\der u \right) \nonumber \\
&& \qquad =   \int\limits_0^T \dfrac{\widetilde W(T-u)}{\widetilde W(T)} \e^{-\eta u} \overline{\Pi}(u) \der u =  \int\limits_0^T \dfrac{\e^{\eta(T-u)}\widetilde W (T-u)}{\e^{\eta T}\widetilde W(T)}  \overline{\Pi}(u) \der u \nonumber \\
&&\qquad =  \dfrac{1}{W(T)}\int\limits_0^T  W(T-u) \overline{\Pi}(u) \der u = 1 - \dfrac{1}{W(T)}  \nonumber 
\end{eqnarray}

The second statement can also be established thanks to Proposition~\ref{prop:duality} (ii). This duality property, together with equations~(\ref{eq:exit-two-side}) and (\ref{eq:prob-cond-hit-0}), imply that
\[
\widetilde\Pro_\top\left(\xi_T=0\right) = \widetilde P_\top\left(\tau_0<\tau_T^+\right) = P_T\left(\tau_T^+<\tau_0\right) = 1 - \dfrac{1}{W(T)}
\] 
which is the expected result.
\end{proof}
\medskip

\begin{proof}[Proof of Lemma~\ref{lemma:last-exc}]
Let $\varepsilon\coloneqq Y\circ k_{\tau_0} = (\Y_t,0\leq t < \tau_0)$, we want to prove that $\varepsilon$ and $\rho\circ \varepsilon$ have the same law under the probability measure $P_T\left(\cdot\middle\vert \tau_0<\tau_T^+\right)$. 

Define
\[
\varsigma_T\coloneqq \sup \{t\in[0,\tau_0): Y_t = T\}.
\]
Since the process makes no negative jumps, under $P_0(\cdot\vert\tau_0<+\infty)$ the events $\varsigma_T<+\infty$ and $\tau_T^+<\tau_0$ a.s. coincide. Hence, if we condition the excursion $\varepsilon$ to enter the interval $[T,+\infty)$, and $\tau_0$ to be finite, the law of the excursion shifted  to this last passage at $T$, that is 
\begin{equation}\label{eq:aster}
P_0(\varepsilon\circ \theta_{\varsigma_T} \in \cdot\vert \tau_T^+<\tau_0<+\infty) =P_T\left(\cdot\middle\vert \tau_0<\tau_T^+\right)\circ k^{-1}_{\tau_0}. 
\end{equation}

On the other hand, we know from Proposition~\ref{prop:duality} that the probability measure of $\varepsilon$ starting at 0 and conditional on $\{\tau_0<+\infty\}$ is invariant under space-time-reversal, meaning that for every bounded measurable function $F$ we have
\begin{eqnarray} 
E_0\left[F(\varepsilon)\middle\vert \tau_0<+\infty\right] = E_0\left[F(\rho\circ\varepsilon)\middle\vert \tau_0<+\infty\right] \nonumber
\end{eqnarray}
In particular, for any $f$ also bounded and measurable, take
\[
F(\varepsilon) = f(\varepsilon\circ \theta_{\varsigma_T(\varepsilon)})1_{\{\sup \varepsilon \geq T\}}.
\]
If we now apply this function to the reversed excursion, it is not hard to see that, 
$P_0(\cdot\vert\tau_0<+\infty)$ a.s. we have, $\{\sup_{[0,\tau_0]}\rho\circ\varepsilon \geq T\} =\{\inf_{[0,\tau_0]}\varepsilon \leq  -T\}$ and $\varsigma_T(\rho\circ\varepsilon) = \tau_0(\varepsilon) - \tau_{-T}(\varepsilon)$. From the definition of space-time-reversal and shifting operators it also follows that
\[
(\rho\circ\varepsilon)\circ \theta_{\varsigma_T(\rho\circ\varepsilon)} = \left(-\varepsilon((\tau_0-t)-)\right)\circ \theta_{\tau_0-\tau_{-T}}(\varepsilon) = -\varepsilon((\tau_{-T}-t)-) = \rho\circ (\varepsilon\circ k_{\tau_{-T}}).
\]
Hence for any bounded measurable $f$ it holds that
\begin{eqnarray} 
E_0\left[f(\varepsilon\circ \theta_{\varsigma_T})1_{\{\sup \varepsilon \geq T\}} \middle\vert \tau_0<+\infty\right] = E_0\left[f\left(\rho\circ (\varepsilon\circ k_{\tau_{-T}})\right) 1_{\{\inf \varepsilon \leq -T\}} \middle\vert \tau_0<+\infty\right], \nonumber
\end{eqnarray}
and in particular for $f \equiv 1$ we have
\begin{eqnarray}
P_0\left(\sup\varepsilon \geq T\middle\vert \tau_0<+\infty\right) = P_0\left(\inf\varepsilon \leq -T\middle\vert \tau_0<+\infty\right). \nonumber 
\end{eqnarray}
Combining these two equations and using that $\{\sup\varepsilon\geq T\} = \{\tau_T^+< \tau_0\}$ and $\{\inf\varepsilon\leq -T\}=\{\tau_{-T}< \tau_0^+\}$ a.s. under $P_0(\cdot\vert\tau_0<+\infty)$, we have
\begin{equation}\label{eq:aster-2}
P_0\left(\varepsilon\circ \theta_{\varsigma_T}\middle\vert \tau_T^+< \tau_0\right) = P_0\left(\rho\circ (\varepsilon\circ k_{\tau_{-T}})\middle\vert \tau_{-T}< \tau_0^+\right).
\end{equation}

According to (\ref{eq:aster}), the left-hand-side in (\ref{eq:aster-2}) equals $P_T\left(\cdot\middle\vert \tau_0<\tau_T^+\right)\circ k^{-1}_{\tau_0}$.
Finally, the conclusion comes from the following consequences of the strong Markov property. First, the excursion $\varepsilon\circ k_{\tau_{-T}}$ has the same law under $P_0(\cdot\vert \tau_{-T}<\tau_0,\tau_0<+\infty)$ as under $P_0(\cdot\vert \tau_{-T}<\tau_0^+)$ . And $\rho\circ(\varepsilon\circ k_{\tau_{-T}})$ under $P_0(\cdot\vert\tau_{-T}<\tau_0^+)$ has the same law as $\rho\circ(\varepsilon\circ k_{\tau_{0}})$ under $P_T(\cdot\vert\tau_{0}<\tau_T^+)$.
\end{proof}

\medskip

\begin{proof}[Proof of Lemma~\ref{lemma:funct-law}]
We need to prove that for any $a>0$, $x\in[0,a]$ and $\Lambda \in \mathscr{F}_{\tau_a}$ the following identity holds
\[
\widetilde P_x\left(\Lambda,\tau_a^+<\tau_0\right) =
P_x\left(\Lambda, \tau_a^+<\tau_0,\tau_a<+\infty\right) \e^{-\eta(a-x)}
\]

This is trivial when $\eta = 0$, so we will suppose from now $\eta>0$.
Since the process makes only positive jumps, we have $\{\tau_a<\tau_0\}=\{\tau_a^+<\tau_0\}$, $\widetilde P_x$ a.s., so we can start looking at
\begin{eqnarray} 
\widetilde P_x \left(\Lambda, \tau_a<\tau_0\right) 
= \widetilde E_x\left[ \widetilde E_x\left[ \1_{\{\Lambda, \tau_a<\tau_0\}}\middle\vert \mathscr{F}_{t\wedge\tau_a} \right] \right]
= E_x\left[ \widetilde E_x\left[ \1_{\{\Lambda, \tau_a<\tau_0\}}\middle\vert \mathscr{F}_{t\wedge\tau_a} \right]  \frac{\e^{-\eta \Y_{t\wedge\tau_a}}}{\e^{-\eta x}}\right] \nonumber 
\end{eqnarray}
and subsequently by the strong Markov property the last term equals
\begin{eqnarray} 
= E_x\left[ \1_{\{\Lambda, t\wedge\tau_a<\tau_0\}} \widetilde P_{\Y_{t\wedge \tau_a}}\left(\tau_a<\tau_0 \right)  \frac{\e^{-\eta \Y_{t\wedge\tau_a}}}{\e^{-\eta x}}\right] \nonumber 
\end{eqnarray}
We notice that the process $\Y_{t\wedge\tau_a}$ stays positive on $\{t\wedge\tau_a<\tau_0\}$, hence 
\[
\1_{\{t\wedge \tau_a<\tau_0\}}\e^{-\eta \Y_{t\wedge\tau_a}}\leq 1.
\]
Besides, when $\tau_a=+\infty$, this is the same as $\e^{-\eta \Y_t}$, which tends to 0 when $t\rightarrow +\infty$ because the process drifts to $+\infty$ under $P$ when $\eta>0$. The other terms are bounded by 1, so the dominated convergence theorem applies and we have
\begin{equation*}
\widetilde P_x \left(\Lambda, \tau_a<\tau_0\right)  \xrightarrow[t\rightarrow +\infty]{}  E_x\left[ \1_{\{\Lambda, \tau_a<\tau_0\}}  \frac{\e^{-\eta a}}{\e^{-\eta x}}\1_{\{\tau_a<+\infty\}}\right],
\end{equation*}
which concludes the proof.
\end{proof}

\section*{Acknowledgements}
This work was supported by grants from Région Ile-de-France and ANR. The authors also thank the Center for Interdisciplinary Research in Biology (CIRB) for funding.

\bibliographystyle{alea3}
\bibliography{biblio-utf8}

\end{document}